\documentclass[letterpaper, 11pt]{article}
\usepackage[utf8]{inputenc}
\usepackage[margin=1in]{geometry}
\usepackage{amsmath}
\usepackage{amsfonts}
\usepackage{color}
\usepackage{dirtytalk}
\usepackage{amssymb,amsthm}
\usepackage{hyperref}
\usepackage[pagewise]{lineno}
\numberwithin{equation}{section}
\theoremstyle{plain}
\newtheorem{theorem}{Theorem}[section]
\newtheorem*{mainthm}{Main Theorem}
\newtheorem{lemma}[theorem]{Lemma}

\newtheorem{proposition}[theorem]{Proposition}

\theoremstyle{definition}
\newtheorem{definition}[theorem]{Definition}
\newtheorem{remark}[theorem]{Remark}

\usepackage{graphicx}

\title{Approximation of $N$-player stochastic games with singular controls by mean field games\\{\em Honoring Prof. Jin Ma's 65th birthday}}
\author{Haoyang Cao\footnote{CMAP, \'Ecole Polytechnique, Route de Saclay, 91128 Palaiseau Cedex, France, \url{haoyang.cao@polytechnique.edu}} \and Xin Guo\footnote{Department of Industrial Engineering and Operations Research, University of California at Berkeley, Berkeley, CA 94720, USA, \url{xinguo@berkeley.edu}} \and Joon Seok Lee\footnote{Laboratoire de Probabilit\'es et Mod\`eles Al\'eatoires, CNRS,  UMR 7599, Universit\'e Paris Diderot, \url{delinbetances@gmail.com}}}
\date{\today}

\begin{document}

\maketitle

\begin{abstract}
	This paper establishes that a class of $N$-player stochastic games with singular controls, either of bounded velocity or of finite variation, can both be approximated by mean field games (MFGs) with singular controls of bounded velocity. More specifically, it shows (i) the optimal control to an MFG with singular controls of a bounded velocity $\theta$ is shown to be an $\epsilon_N$-NE to an $N$-player game with singular controls of the bounded velocity, with $\epsilon_N = O(\frac{1}{\sqrt{N}})$, and (ii) the optimal control to this MFG is an $(\epsilon_N + \epsilon_{\theta})$-NE to an $N$-player game with singular controls of finite variation, where $\epsilon_{\theta}$ is an error term that depends on $\theta$. This work generalizes the classical result on approximation $N$-player games by MFGs, by allowing for discontinuous controls.
\end{abstract}

\section{Introduction}
{$N$}{-player} non-zero-sum stochastic games are notoriously hard to analyze. The theory of Mean Field Games (MFGs), pioneered by \cite{LL2007} and \cite{HMC2006}, presents a powerful approach to study stochastic games of a large population with small interactions. (See the lecture notes and books \cite{BFY2013}, \cite{CDLL2015}, \cite{CarmonaDelarue}, \cite{GLL2010}, and the references therein for more details on MFGs).
The key idea behind MFGs is to avoid directly analyzing the difficult $N$-player stochastic games, and  instead to approximate the dynamics and the objective function via the notion of population's probability distribution flows, a.k.a., mean information processes.  This idea is feasible if an MFG can approximate the corresponding $N$-player game, under proper criteria. The seminal work of  \cite{HMC2006} demonstrated that this is indeed the case, and showed that the value function of an $N$-player game under the criterion of Nash equilibrium (NE) can be approximated by the value function of an associated MFG with an error of order $\frac{1}{\sqrt{N}}$. There are other on higher order error analyses through the central limit theorem and the large deviation principle for MFGs. For instance,  \cite{DLR2018b} and \cite{DLR2019} studied diffusion-based models with common noise via the coupling approach, and \cite{BC2018} and \cite{CP2018}  analyzed finite state space models without common noise using master equations. As such, MFGs provide an elegant and analytically feasible framework to approximate $N$-player stochastic games. 
 
All existing works on approximation of  $N$-player stochastic games by MFGs are established  within the framework of regular controls where controls are absolutely continuous. However, most control problems from engineering and economics are not absolutely continuous, or even continuous. A natural question is, will this relation between the MFG and the $N$-player game hold  when  controls may not be continuous? 
 
The focus of this paper is to establish, within the singular control framework,  the approximation of  $N$-player stochastic games  by their corresponding MFGs.

{\bf MFGs and stochastic games with singular controls.}
Compared with regular controls, singular controls provide a more general and natural  mathematical framework where both the controls and the state space may be discontinuous.
However,  it is well documented that analysis for singular controls is much harder than for regular controls.
From a PDE perspective, the associated fully nonlinear PDE is coupled with possibly state and time dependent gradient constraints. 
From a control perspective, the Hamiltonian for singular controls of  finite variation  diverges \cite{Pham2009} and the standard stochastic maximal principle fails; even in the case of bounded velocity, the Hamiltonian is discontinuous. In contrast, the existence of solutions to MFGs relies on the assumption that the Hamiltonian $H(x,p)$ has sufficient regularity, especially with respect to $p$. For instance, \cite{LL2007} assumed  that $H$ is of class $\mathcal{C}^1$ in $p$,  and 
\cite{Cardaliaguet2013} assumed  that $H$ is of class $\mathcal{C}^2$ and that the second-order derivative with respect to $p$ is Lipschitz continuous. The exception is \cite{Lacker2015}, which established in a general framework  the existence of Markovian equilibrium solutions when controls are continuous but may not be Lipschitz.  
\cite{FH2016} adopted the notion of relaxed controls for the existence of solution to MFGs with singular controls and established its approximation by MFGs with regular controls.

Nevertheless, the question remains as to whether $N$-player games can be approximated by MFGs,  when controls may not be absolutely continuous.

{\bf Our work.} There are two types of singular controls, namely, singular controls of finite variation and singular controls of bounded velocity.  This paper establishes that $N$-player stochastic games with singular controls, {\it both} of finite variation and of bounded velocity, can be approximated under the NE criterion by MFGs with singular controls of bounded velocity.  This result suggests that one may  completely circumvent the more difficult MFGs of singular controls of finite variation, when analyzing stochastic games of singular type, and instead  focus on  singular controls games of bounded velocity. 
  
Indeed,  singular controls of bounded velocity share some nice properties with regular controls and are easier to analyze than singular controls of finite variation. This conviction underlines the main idea in our analysis of the relation between   MFGs and  the associated $N$-player stochastic games. The analysis starts with two basic components. First is the relationship between the underlying singular control problems,  bounded velocity vs finite variation.  Theorem~\ref{thetainfty} shows  that under proper assumptions, the value function of the former converges to that of the latter. Second is on the existence, the uniqueness, and the regularity  for the solution to the MFG with singular controls of bounded velocity, established in Theorem~\ref{mainthm}. These two ingredients lead to the main theorem  on approximation of MFGs to the corresponding $N$-player games. Specifically, (i) given a bounded velocity $\theta$, the optimal control to the MFG with singular controls of bounded velocity   is an $\epsilon_N$-NE to an $N$-player game with singular controls of bounded velocity with $\epsilon_N = O(\frac{1}{\sqrt{N}})$, and (ii) the optimal control to the MFG is  an  $(\epsilon_N + \epsilon_{\theta})$-NE to an $N$-player game with singular controls of finite variation, where $\epsilon_{\theta}$ is an error term that depends on $\theta$.

{\bf Other related work.}
There are earlier works relating singular controls with bounded velocity and with finite variation. For instance, exploiting this relation enables \cite{MT1989} to establish the existence of the optimal singular control of finite variation for a controlled Brownian motion. This relation is also analyzed in \cite{HPY2016} for a monotone follower type of singular controls. None of these works is in a game setting. 
Moreover, to establish the relation between MFGs and $N$-player games in a singular control framework, one needs more explicit construction for the optimal control policies. 

A Markov chain based approximation approach was proposed in \cite{BBC2018} for numerically solving MFGs with reflecting barriers and showed its convergence. Then in \cite{DF2018} 
it was shown that,  under the notion of weak (distributional) NE,  $N$-player stochastic games with singular controls of finite variation can be approximated by that of bounded velocity, if the set of Nash equilibria for the latter is relatively compact under an appropriate topology. The focus and approach of these works are different from ours. 
 
Finally, the existence of Markovian NE solution for MFGs in  Theorem \ref{mainthm} was established in  \cite{Lacker2015} in a more general class of MFGs. His approach is sophisticated and consists of two main steps. The first step is the existence of 
a weak solution under the convexity assumption, and the second step is to go through a measurable selection argument to show that this weak solution is in fact the desirable one. Our approach is  to directly construct the Markov NE using the fixed point approach, based on the special structure of the game.  This  yields more explicit solution structure with additional regularity properties, which are necessary for  the subsequent analysis to connect  MFGs and the associated $N$-player games.

\section{Problem formulations and main results}
\label{setup}

We start with   $(\Omega, \mathcal{F}, \mathbb{F} =(\mathcal{F}_t)_{0\leq t \leq \infty}, P)$  a probability space in which  {$\mathbf{W}^i = \{W_t^i\}_{0\leq t\leq \infty}$} are i.i.d. standard Brownian motion with $i=1,\ldots,N<\infty$. Let 
$\mathcal{P}  (\mathbb{R}) $ be the set of all probability measures on  $\mathbb{R}$, and
 $\mathcal{P}_p (\mathbb{R}) $  be the set of all probability measures of $p$th order on  $\mathbb{R}$. That is 
$$\mathcal{P}_p (\mathbb{R}) = \biggl\lbrace \mu \in \mathcal{P}  (\mathbb{R}) \biggl| \left(\int_\mathbb{R} |x|^p  \mu( dx)\right )^\frac{1}{p}  < \infty \biggr\rbrace.$$
To define the flow of probability measures $\{\mu_t\}_{t\ge 0}$, let us recall the  $p$th order Wasserstein metric on $	\mathcal{P}_p(\mathbb{R})$   defined as 
$$D^p(\mu,\mu')=  \inf_{\tilde{\mu} \in \Gamma(\mu, \mu') }   \limits  \left(\int_{\mathbb{R}\times\mathbb{R}}  |y-y'|^p  \tilde{\mu} (dy,dy')  \right)^{\frac{1}{p}},$$ 
where  $\Gamma(\mu, \mu')$  is the set of all coupling of $\mu $ and $\mu' $.
Denote $C([0, T], \mathcal{P}_2 (\mathbb{R}))$ for all continuous mappings from $[0, T]$ to $\mathcal{P}_2 (\mathbb{R})$. Then  $\mathcal{M}_{[0,T]} \subset C([0, T],\mathcal{P}_2 (\mathbb{R}))$
is a class of flows of probability measures  such that there exists a positive constant $c$ so that    
\begin{align*}
\mathcal{M}_{[0,T]} = \biggl\lbrace \{ \mu_t\}_{0\le t\le T} ~\biggl|~ \sup_{s\neq t}\frac{ D^1(\mu_t,\mu_s) }{|t-s|^{\frac{1}{2}}} \leq c,  \sup_{t \in [0,T]} \int_\mathbb{R} |x|^2 \mu_t(dx) \leq c\biggr\rbrace.
\end{align*}
$\mathcal{M}_{[0,T]}$ is a metric space endowed with the metric 
\begin{align}\label{metric1}
d_\mathcal{M}\biggl(\{\mu_t\}_{0\le t\le T},\{\mu_t'\}_{0\le t\le T}\biggr) = \sup_{0\le t\le T} D^2(\mu_t,\mu_t').
\end{align} 

Throughout, we will use $Lip(\psi)$ as a Lipschitz coefficient of $\psi$ for any given Lipschitz function $\psi$. That is, $|\psi(x)-\psi(y) | \leq Lip(\psi) |x-y|$ for any $x,y \in \mathbb{R}$. {For any $\psi(x)\in\mathcal{C}^2$, we will use
\[\mathcal{L} \psi (x) = b(x) \partial_x \psi (x)+\frac{1}{2} \sigma^2 (x) \partial_{xx} \psi (x),\]
for the infinitesimal generator for any stochastic process 
\[dx_t = b(x_t) dt + \sigma (x_t) dW_t,\]
where $b$ and $\sigma$ are Lipschitz continuous and of linear growth; we say that a function $f$ is of a polynomial growth if  $|f (x)| \leq c(|x|^k+1) $ for some positive constant $c$ and $k\in\mathbb{N}$, for all $x\in\mathbb R$.}

\subsection{Problems of N-player stochastic games and MFGs} 
 
{\bf ${N}$-player game with singular controls of finite variation.}  Fix a  time $T <\infty$ and suppose that there are $N$ {rational and indistinguishable} players in the game.  Denote $ \{x_t^i\}_{s \leq t \leq T}$ as the state process in $\mathbb{R}$ for player $i$ ($i = 1, \ldots, N$), with $x_{s-}^i=x^i$ starting from time $s\in [0,T]$.
Now assume that  the  dynamics of $\{x_t^i\}$ follows,  for $s\le t \le T$,
\begin{align}  \label{nSDE}
 dx_t^i =  \frac{1}{N}\sum_{j=1}^N b_0(x_t^i,x_t^j)    dt + \sigma  dW_t^i +d\xi_t^{i+}-d\xi_t^{i-}, \ \ \ x_{s-}^i=x^i,
 \end{align}
where $b_0: \mathbb{R} \times \mathbb{R} \rightarrow \mathbb{R}$ is bounded, Lipschitz continuous, and $\sigma$ is a positive constant.
Here $   \xi_\cdot^i  = (\xi_\cdot^{i+},\xi_\cdot^{i-}) $ is the control by player  $i$ with $ (\xi_\cdot^{i+},\xi_\cdot^{i-}) $ nondecreasing, c\`adl\`ag,  $\xi_{s-}^{i+}=\xi_{s-}^{i-} = 0$, $\mathbb{E} \biggl [\int_s^T d\xi_t^{i+}  \biggr]< \infty, $ and $\mathbb{E} \biggl[\int_s^T d\xi_t^{i-}  \biggr] < \infty $. 
 
Given Eqn. (\ref{nSDE}), the objective of player $i$ is to minimize, over an appropriate control set $\mathcal{U}^N$, her cost function  $J^{i,N}(s,x^i , \xi_\cdot^{i+},\xi_\cdot^{i-};\xi_\cdot^{ -i} )$. That is

\begin{align}\label{Nsingular}\tag{N-FV}
\begin{split}
\inf_{ (\xi_\cdot^{i +}, \xi_\cdot^{i -}) \in \mathcal{U}^N} J^{i,N}(s,x^i,\xi_\cdot^{i+},\xi_\cdot^{i-};\xi_\cdot^{ -i}) & = \inf_{ (\xi_\cdot^{i +}, \xi_\cdot^{i-}) \in \mathcal{U}^N} \mathbb{E}  \left[ \int_s^T \frac{1}{N}  \sum_{j=1}^N f_0(x_t^i, x_t^j)dt + \gamma_1 d\xi_t^{i+}+ \gamma_2 d\xi_t^{i-}  \right].
 \end{split}
\end{align} 

Here  $\xi_\cdot^{ -i}=\{ (\xi_\cdot^{j +}, \xi_\cdot^{j -})\}_{j=1, j \neq i }^N$  denotes the set of controls  for all the players except for player $i$, the cost function $f_0: \mathbb{R} \times \mathbb{R} \rightarrow \mathbb{R}$ is Lipschitz continuous,  $\gamma_1$ and $\gamma_2  $ are constants, and  

\begin{align*}
\mathcal{U}^N = \biggl\lbrace  (\xi_\cdot^+,\xi_\cdot^-) ~\biggl|~ &  \xi_t^+ \text{ and }\xi_t^- \space \text{ are } \mathcal{F}_t^{(x_t^1,\ldots, x_t^N)} \text{-adapted, c\'adl\'ag, nondecreasing, } \\ & 
 \xi_{s-}^+ = \xi_{s-}^- =0, \mathbb{E} \biggl[\int_s^T d\xi_t^+ \biggl]<\infty, \text{ and } \mathbb{E} \biggl[\int_s^T d\xi_t^-  \biggl] < \infty ,\text{ for } 0\le s\le t\le T  \biggl\rbrace,
\end{align*} 
with  {$\{ \mathcal{F}_t^{(x_t^1,\ldots, x_t^N)}\}_{s\le t\le T}$} the natural filtration of $\{x_t^1,\ldots, x_t^N\}_{s\le t\le T}.$   

{\bf $N$-player game with singular controls of bounded velocity.}
If one restricts  the controls $(\xi_\cdot^{i+},\xi_\cdot^{i-})$ to be with a bounded velocity such that  for a given constant $\theta >0 $,
$$d\xi_t^{i+} = \dot{\xi}_t^{i+}dt, \ \ \  d\xi_t^{i-} = \dot{\xi}_t^{i-} dt, $$ with $0 \leq \dot{\xi}_t^{i+}, \dot{\xi}_t^{i-} \leq \theta$.
Then  game (\ref{Nsingular}) becomes 

\begin{align}\label{Nbound}\tag{N-BD}
\inf_{(\xi_\cdot^{i +}, \xi_\cdot^{i -}) \in \mathcal{U}_{\theta}^N}  J^{i,N}_{\theta} (s,x^i,\xi_\cdot^{i+},\xi_\cdot^{i-};\xi_\cdot^{ -i}  ) & = \inf_{ (\xi_\cdot^{i +}, \xi_\cdot^{i -}) \in \mathcal{U}_\theta^N}  \mathbb{E} \left[ \int_s^T \frac{1}{N}  \sum_{j=1}^N f_0(x_t^i, x_t^j)dt + \gamma_1 \dot{\xi}_t^{i+}dt +\gamma_2\dot{\xi}_t^{i-}dt  \right],
 \\ \text{subject to } \quad  dx_t^i &=  \frac{1}{N}\sum_{j=1}^N b_0(x_t^i,x_t^j)    dt + \sigma  dW_t^i +\dot{\xi}_t^{i+}dt-\dot{\xi}_t^{i-}dt, \ \ \ x_s^i=x^i.
\end{align}

Here the admissible set is given by
$$\mathcal{U}_{\theta}^N = \biggl\lbrace (\xi_\cdot^+,\xi_\cdot^-) ~\biggl|~ (\xi_\cdot^+,\xi_\cdot^-) \in \mathcal{U}^N , 0\le \dot{\xi}_t^+,\dot{\xi}_t^- \le \theta, \text{ for }  0 \le s\le t\le T \biggl\rbrace.$$
 
There are several criteria to analyze stochastic games. Two standard ones are the Pareto optimality and the {Nash equilibrium (NE)}.  In this paper we will focus on  NE. Depending on the problem setting and in particular the admissible controls, there are several forms of Nash equilibria (NEs), including the open loop NE, the closed loop NE, and the closed loop in feedback form NE (a.k.a., the Markovian NE). Throughout the paper, we will consider the Markovian NE. 
Markovian NE means that the controls  are deterministic functions of time $t$,  current state $x_t$, and a fixed measure $\mu_t$. More precisely,

\begin{definition}[Markovian $\epsilon$-Nash equilibrium to (\ref{Nsingular})] 
A Markovian control $(\xi_\cdot^{i*+}, \xi_\cdot^{i*-}) \in \mathcal{U}^N$ for $i = 1,\ldots, N$ is a  \emph{Markovian $\epsilon$-Nash equilibrium} to \emph{(\ref{Nsingular})} if for any $i \in \{1, \ldots, N\}$, any $(s,x) \in [0,T]\times \mathbb{R}$  and any Markovian $(\xi_\cdot^{i'+},\xi_\cdot^{i'-}) \in \mathcal{U}^N$,  $$E_{x_{s-}^{N}}\left[J^{i,N} (s,x_{s-}^{N},\xi_\cdot^{i'+},\xi_\cdot^{i'-};\xi_\cdot^{*-i} )\right] \geq E_{x_{s-}^{N}}\left[J^{i,N} (s,x_{s-}^N,\xi_\cdot^{i*+},\xi_\cdot^{i*-};\xi_\cdot^{*-i} )\right] -\epsilon.$$  
\end{definition}

{\begin{definition}[Markovian $\epsilon$-Nash equilibrium to (\ref{Nbound})] 
A Markovian control $(\xi_\cdot^{i*+}, \xi_\cdot^{i*-}) \in \mathcal{U}_{\theta}^N$ for $i = 1,\ldots, N$ is a  \emph{Markovian $\epsilon$-Nash equilibrium} to \emph{(\ref{Nbound})} if for any $i \in \{1, \ldots, N\}$, any $(s,x) \in [0,T]\times \mathbb{R}$  and any Markovian $(\xi_\cdot^{i'+},\xi_\cdot^{i'-}) \in \mathcal{U}_{\theta}^N$,  $$E_{x_{s,\theta}^N}\left[J^{i,N}_{\theta} (s,x_{s,\theta}^N,\xi_\cdot^{i'+},\xi_\cdot^{i'-};\xi_\cdot^{*-i} )\right] \geq E_{x_{s,\theta}^N}\left[J^{i,N}_{\theta} (s,x_{s,\theta}^N,\xi_\cdot^{i*+},\xi_\cdot^{i*-};\xi_\cdot^{*-i} ) \right]-\epsilon.$$  
\end{definition}}

We will show that both $N$-player games, game (\ref{Nbound}) and game  (\ref{Nsingular}), can be approximated by MFGs with singular controls of bounded velocity, as introduced below.

{\bf MFGs with singular controls of bounded velocity.}
Assume that all $N$ players are identical. That is, for each time $t \in [0,T]$, all $x_t^i$ have the same probability distribution. Define $ \mu_t = \lim_{N \rightarrow \infty} \limits \frac{1}{N} \sum_{i =1}^N \limits \delta_{x_t^i}$ as  a limit of the empirical distributions of $x_t^i$. Then, according to SLLN, as $N\rightarrow \infty$,  
\begin{align*}
&\frac{1}{N}\sum_{j=1}^N b_0(x_t ,x_t^j) \rightarrow  \int_\mathbb{R} b_0(x_t ,y) \mu_t(dy)=b(x_t ,\mu_t) , \ \ \mathbb{P}-a.s.,
\\ &\frac{1}{N}\sum_{j=1}^N f_0(x_t , x_t^j) \rightarrow \int_\mathbb{R} f_0(x_t , y) \mu_t(dy) =f(x_t , \mu_t),\ \ \mathbb{P}-a.s.,
\end{align*} 
subject to appropriate technical conditions. Here $b, f:\mathbb{R} \times  \mathcal{P}_1(\mathbb{R}) \rightarrow \mathbb{R}$  are functions satisfying  assumptions to be specified later.  That is, instead of game (\ref{Nbound}), 
one can solve for a pair of control $\{\xi^*_t\}_{t\in[0,T]}$ and mean information $\{\mu^*_t\}_{t\in[0,T]}$ such that 
\begin{enumerate}
    \item Under $\{\mu^*_t\}_{t\in[0,T]}$, $\{\xi^*_t\}_{t\in[0,T]}=\{(\xi_t^{*,+},\xi_t^{*-})\}_{t\in[0,T]}$ is an optimal strategy for
    \begin{align}\label{MFGbounded1}\tag{MFG-BD}
\begin{split}
v_{\theta} (s,  x|\{\mu^*_t\} )  & :=  \inf_{(\xi_\cdot^{ +},  \xi_\cdot^{ -}) \in \mathcal{U}_{\theta}} J_{\theta} (s,  x , \xi_\cdot^{ +}, \xi_\cdot^{ -} |\{\mu^*_t\}) \\&: = \inf_{ (\xi_\cdot^{+}, \xi_\cdot^{ -}) \in \mathcal{U}_{\theta}} \mathbb{E}_{\mu^*_s}\left[ \int_s^T   \left( f(x_t, \mu_t)+ \gamma_1\dot{\xi}_t^{+}+\gamma_2\dot{\xi}_t^{-} \right) dt|x_s=x\right],
\end{split}
\end{align}
subject to  
\begin{equation}\label{dynamics-bdd}
 dx_t  = \left( b(x_t,\mu_t^*) + \dot{\xi}_t^{+}- \dot{\xi}_t^{-} \right) dt + \sigma  dW_t , \quad x^*_s \sim \mu^*_s,  
\end{equation}
\begin{align*}
\mathcal{U}_{\theta} = \biggl\lbrace (\xi_\cdot^+,\xi_\cdot^-) \biggl|&  \xi_t^+\text{ and }\xi_t^-  \text{ are } \mathcal{F}_t^{(x_{t-})} \text{-adapted, c\'adl\'ag, nondecreasing, } 
 \xi_{s}^+=\xi_{s}^-=0, \\& 0\le \dot{\xi}_t^+,\dot{\xi}_t^-\le \theta,\mathbb{E} \biggl[\int_s^T  d\xi_t^+  \biggl] <\infty ,  \text{ and } \mathbb{E} \biggl[\int_s^T  d\xi_t^-  \biggl] < \infty, \text{ for } 0\le s\le t \le T   \biggl\rbrace,
\end{align*}
with $ \{ \mathcal{F}_t^{(x_{t-})} \}_{s\le t\le T} $  the filtration of $ \{(x_{t-})\}_{s\le t\le T} $. 
When $\theta\to \infty$, we simply write $\mathcal{U}$ instead of  $\mathcal{U}_{\infty}$ for notational simplicity.
\item $\mu_t^*$ is the probability distribution of $x_t^*$ which is given by 
$$  dx_t^{*} = \left( b(x_t^{*},\mu_t^{*}) + \dot{\xi}_t^{*+}- \dot{\xi}_t^{*-} \right) dt + \sigma  dW_t ,\quad s\le t\le T,  \quad x_s^{*} \sim \mu_s^*. $$
\end{enumerate}
Such a pair $(\xi^{*,+}_\cdot,\xi^{*,-}_\cdot)\in\mathcal{U}_{\theta}$ and $\{\mu_t^*\}\in\mathcal{M}_{[0,T]}$ constitute  a solution of \eqref{MFGbounded1}.
\begin{remark}
\label{remark-fixedinitial}
Note here the game value $v_{\theta}(s, x|\{\mu^*_t\})=\inf_{(\xi_\cdot^{ +},  \xi_\cdot^{ -}) \in \mathcal{U}_{\theta}} J_{\theta} (s,  x , \xi_\cdot^{ +}, \xi_\cdot^{ -} |\{\mu^*_t\})$ with $x_s^*=x$ being a sample from $\mu^*_s$. An alternative definition of the game is to solve  
 $\tilde{v}_{\theta}(s, \mu^*_s)$ with $\tilde{v}_{\theta}(s, \mu^*_s)=\mathbb{E}_{\mu^*_s}[v_{\theta}(s, x_s)]$. This game value can be easily  recovered from 
 $v_{\theta}(s,x)$. (See also \cite{GX2019} and \cite[Section 2.2.2]{LZ2018} for a similar set up.)
\end{remark}
For ease of exposition, we will use the following notion of control function, for a fixed $\mu_t$. 
\begin{definition} [Control function]
A control of  bounded velocity $\xi_t$ is called  \emph{Markovian} 
if $ d\xi_t = \dot{\xi}_t dt =   \varphi (t,x_t|\{\mu_t\}) dt$ for some function $\varphi:[0,T]\times \mathbb{R}\rightarrow \mathbb{R}$. $\varphi(t,x_t|\{\mu_t\})$ is called the \emph{control function} for the fixed $\{\mu_t\}$. A control of a finite variation $\xi_t$ is called \emph{Markovian}  if  $ d\xi_t   =  d\varphi(t,x_t|\{\mu_t\})$  for some function $\varphi$. $\varphi$ is called the \emph{control function} for the fixed $\{\mu_t\}$. 
\end{definition}

\subsection{Main results}
The main results are derived based on the following assumptions. 
\begin{itemize}
\item[(A1).] $b_0(x,y)$ and $f_0(x,y)$ are Lipschitz continuous in both $x$ and $y$. That is, $|b_0(x_1,y_1)-b_0(x_2,y_2)|\leq Lip(b_0)(|x_1-x_2|+|y_1-y_2|)$ and $|f_0(x_1,y_1)-f_0(x_2,y_2)|\leq Lip(f_0)(|x_1-x_2|+|y_1-y_2|)$ for some $Lip(b_0),Lip(f_0)>0$. Moreover, $|b_0(x,y)|\leq c_1$ for some $c_1$. $b(x,\mu)$ and $ f(x,\mu) $ are Lipschitz continuous in $x$ and $\mu$, and $b(x,\mu)$ is bounded. That is, $| b(x_1,\mu^1)- b(x_2,\mu^2)| \le Lip(b)( |x_1-x_2| + D^1(\mu^1,\mu^2)) $ for some $Lip(b) >0$, and $| f(x_1,\mu^1)- f(x_2,\mu^2)| \le Lip(f)( |x_1-x_2| + D^1(\mu^1,\mu^2)) $ for some $Lip(f) >0$, and $|b(x,\mu)| \le c_2$ for some $c_2$. 
\item[(A2).] $f(x,\mu) $ has a first-order derivative in $x$ with   $f(x,\mu)$ and $\partial_x f(x,\mu)$ satisfying the polynomial growth condition. Moreover, for any fixed $\mu \in  \mathcal{P}_2(\mathbb{R})$, $f(x,\mu)$ is convex and nonlinear in $x$. Moreover, there exists some constant $c_f$ satisfying $|f(x,\mu)| \leq c_f\biggl(1 + |x|^2+ \int_\mathbb{R} y^2 \mu(dy)   \biggl)$ for any $x\in \mathbb{R}, \mu \in  \mathcal{P}_2(\mathbb{R})$.
 Note that this assumption is well-posed: by definition of $\mathcal{M}_{[0,T]}$, $\mu \in \mathcal{P}_2$.

\item[(A3).]  $b(x,\mu) $ has first- and second-order derivatives with respect to $x$ with uniformly continuous and bounded derivatives in $x$.
\item[(A4).] $-\gamma_1<\gamma_2$.
This ensures the finiteness of the value function. Indeed, take game (\ref{Nsingular}) with $ -\gamma_1	 > \gamma_2$. Then, letting $d\xi_t^{i+} = d\xi_t^{i-} = M $ and $M \rightarrow \infty$, we will have $J^{i,N}\rightarrow -\infty$. 

\item[(A5).] (Monotonicity of the cost function) $f$ satisfies either
\begin{align*}
\mbox{(i).} \int_\mathbb{R} (f(x,\mu^1) - f(x,\mu^2)) (\mu^1 -\mu^2) (dx) \geq 0, \text{ for any } \mu^1 , \mu^2 \in \mathcal{P}_2 (\mathbb{R}) ,
\end{align*}
and $H(x,p ) = \inf_{\dot{\xi}^+,\dot{\xi}^- \in [0,\theta] } \limits \{ (\dot{\xi}^+ -\dot{\xi}^- )p +\gamma_1 \dot{\xi}^+ + \gamma_2 \dot{\xi}^- \} $ satisfies the following condition for any $x,p,q \in \mathbb{R}$ 
\begin{align*}
\text{if }H(x,p+q) - H(x,p) - \partial_p H(x,p) q 	= 0, \text{ then } \partial_p H(x,p+q) = \partial_p H(x,p), \ \ \ \mbox{or}
\end{align*} 
\begin{align*}
\mbox{(ii).} \int_\mathbb{R} (f(x,\mu^1) - f(x,\mu^2)) (\mu^1 -\mu^2) (dx)> 0, \text{ for any } \mu^1 \neq \mu^2 \in \mathcal{P}_2 (\mathbb{R}).
\end{align*}

As in \cite{LL2007, Cardaliaguet2013}, Assumption (A5) is critical  to ensure  the uniqueness  for the solution of (\ref{MFGbounded1}), as will be clear from the proof of Proposition \ref{uniq} for the uniqueness of the fixed point.

\item[(A6).] (Rationality of players) For any control function $ \varphi $, any $t \in [0,T], $ any fixed $\{\mu_t\}$, and any $ x,y \in \mathbb{R}$, $(x-y)\biggl( \varphi (t,x|\{\mu_t\})- \varphi (t,y|\{\mu_t\})\biggl) \leq 0 $. 

Intuitively, this assumption says that the better off the state of an individual player, the less likely the player exercises controls, in order to minimize her cost.  This assumption first appeared in 
\cite{EKPPQ1997} in the analysis of BSDEs.
\end{itemize}

\begin{mainthm}
Assume \emph{(A1)--(A6)}. Then,  
\begin{itemize}
\item[a).] For any fixed $\theta$, the optimal control to \emph{(\ref{MFGbounded1})} is an $\epsilon_{N }$-NE to \emph{(\ref{Nbound})},  given that the distribution of $x_{s,\theta}^N$ at any given initial time $s\in[0,T]$ among $N$ players are permutation invariant. Here  $\epsilon_{N } = O\biggl(\frac{1}{\sqrt{N}}\biggl)$;
\item[b).] The optimal control to \emph{(\ref{MFGbounded1})} is an $(\epsilon_N + \epsilon_\theta)$-NE to  \emph{(\ref{Nsingular})}, given that the distribution of $x_{s}^N$ at any given initial time $s\in[0,T]$ among $N$ players are permutation invariant. Here  $\epsilon_{N } = O\biggl(\frac{1}{\sqrt{N}}\biggl)$, and $\epsilon_\theta \rightarrow 0$ as $\theta \rightarrow \infty $.
\end{itemize} 
\end{mainthm} 

\section{Derivation of the Main Theorem}

The relationship between the stochastic games  (\ref{Nsingular}), (\ref{Nbound}), and (\ref{MFGbounded1}) is built in three steps.  

The first step concerns the analysis of the associated stochastic control problem for (\ref{MFGbounded1}).

\subsection{Control problems}
 
To start, we introduce the underlying stochastic control problems.

 {\bf Control problem of a  bounded velocity.}
Let  $\{\mu_t\} \in \mathcal{M}_{[0,T]}$ be a fixed exogenous flow of probability measures, and consider the following control problem,  
\begin{align} \label{Control}\tag{Control-BD}
\begin{split}
v_{\theta} (s,x |\{\mu_t\}) &  \triangleq \inf_{(\xi_\cdot^{ +},  \xi_\cdot^{ -}) \in \mathcal{U}_{\theta} } J_{\theta}  (s,x , \xi_\cdot^{ +}, \xi_\cdot^{ -}| \{\mu_t\})
\\&  = \inf_{ (\xi_\cdot^{+}, \xi_\cdot^{-})\in \mathcal{U}_{\theta} } \mathbb{E} \left[ \int_s^T   \left( f(x_t, \mu_t)+ \gamma_1 \dot{\xi}_t^{+}+\gamma_2 \dot{\xi}_t^{-} \right) dt  \right],
\end{split}
\end{align}
subject to $dx_t=b(x_t, \mu_t)dt+\sigma dW_t, x_s=x$. 

If  controls are of finite variation, that is, $\theta=\infty$, then we have the following control problem.

{\bf Control problem of finite variation.} 
\begin{align} \label{Control-FV} \tag{Control-FV}
v(s, x |\{\mu_t\} )   & \triangleq \inf_{ (\xi_\cdot^{+}, \xi_\cdot^{-}) \in \mathcal{U}} \mathbb{E} \left[ \int_s^T    f(x_t, \mu_t) dt + \gamma_1 d{\xi}_t^{+}+\gamma_2 d{\xi}_t^{-}    \right], 
\end{align}  
subject to 
\begin{equation*} 
 dx_t  =  b(x_t,\mu_t) dt  + \sigma  dW_t+ d\xi_t^{+}- d\xi_t^{-}  , \quad x_{s-} = x.
\end{equation*}

Note that 
 problem (\ref{Control})
 is a classical stochastic control problem. The associated HJB equation with the terminal condition  is given by
\begin{align}
\begin{split} 
- \partial_t  v_{\theta}  &= \inf_{\dot{\xi}^+,\dot{\xi}^- \in [0,\theta]} \left\lbrace \left( b(x,\mu ) +  (\dot{\xi}^+-\dot{\xi}^-)  \right)  \partial_x v_{\theta}    +  \left(f(x ,\mu ) +\gamma_1\dot{\xi} ^+ + \gamma_2 \dot{\xi}^- \right) \right\rbrace + \frac{\sigma^2}{2}\partial_{xx} v_{\theta} 
\\&=\min \biggl\lbrace  ( \partial_x v_{\theta}+ \gamma_1)\theta,(- \partial_x v_{\theta} + \gamma_2)\theta, 0	\biggl\rbrace  +b(x,\mu )  \partial_x v_{\theta}  +f(x ,\mu )+ \frac{\sigma^2}{2}\partial_{xx} v_{\theta}. 
\\  &\text{with   } v_{\theta} (T, x|\{\mu_t\})=0,  \quad \forall  x \in \mathbb{R} .
\end{split}
\label{HJBHJBHJB}
\end{align}

\begin{proposition}\label{optimization} 
Assume \emph{(A1)--(A4)}. The HJB Eqn. \emph{(\ref{HJBHJBHJB})} has a unique solution $v $ in {$C^{1,2}( [0,T) \times \mathbb{R})\bigcap C( [0,T] \times \mathbb{R})$} with a polynomial growth. Furthermore, this solution is the value function to problem \emph{(\ref{Control})}, and the  corresponding optimal control function is 
\begin{align}\label{optcontrols}
\varphi_\theta (t,x_t|\{\mu_t\})=\dot{\xi}^+_{t,\theta}-\dot{\xi}^-_{t,\theta}   = \left\{
\begin{array}{c l}      
   \theta & \text{if} \quad \partial_x v_{\theta} (t,x_t|\{\mu_t\}) \leq  -\gamma_1,  
  \\ 0 & \text{if}  \quad   -\gamma_1 <  \partial_x v_{\theta} (t,x_t|\{\mu_t\}) <  \gamma_2,
   \\ -\theta & \text{if} \quad  \gamma_2  \leq   \partial_x v_{\theta} (t, x_t|\{\mu_t\}).
\end{array}\right.
\end{align}
 Moreover, the optimal control function $\varphi_\theta (t,x|\{\mu_t\})$ is unique and so is the optimally controlled state process $x_{t,\theta}$ 
with  $$ dx_{t,\theta}  = \biggl( b(x_{t,\theta},\mu_t) +  \varphi_\theta (t,x_{t,\theta}|\{\mu_t\}) \biggl) dt + \sigma  dW_t , \quad x_{s,\theta} = x.$$
\end{proposition}
\begin{proof} By~\cite[Theorem 6.2, Chapter VI]{FR2012}, the HJB Eqn. (\ref{HJBHJBHJB}) has a unique solution $w$ in {$C^{1,2}( [0,T) \times \mathbb{R})\bigcap C( [0,T] \times \mathbb{R})$} with a polynomial growth. 
Standard verification argument  will show that it is the value function to problem (\ref{Control}).
Moreover,  the optimal control function  is
$$\varphi_\theta (t,x_t|\{\mu_t\})= \left\{
\begin{array}{c l}      
   \theta & \text{if} \quad \partial_x v_{\theta}  (t,x_{t,\theta}|\{\mu_t\}) \leq  -\gamma_1,  
  \\ 0 & \text{if}  \quad   -\gamma_1 <  \partial_x v_{\theta}  (t,x_{t,\theta}|\{\mu_t\}) <  \gamma_2,
   \\ -\theta & \text{if} \quad    \gamma_2  \leq   \partial_x v_{\theta}  (t,x_{t,\theta}|\{\mu_t\}).
\end{array}\right. 
$$ 
Now, by Proposition \ref{optimization}, there exists a unique value function $v_{\theta} (t,x|\{\mu_t\}) $ to  problem (\ref{Control}). Furthermore, by (\ref{optcontrols}), the optimal control function $\varphi_\theta (t,x|\{\mu_t\})$ is uniquely determined. Let us prove that the optimally controlled state process $x_{t,\theta}$ exists and is unique. 

For any  given fixed $x_{t,\theta}^n$, consider a mapping $\Phi$ such that  $\Phi(x_{t,\theta}^n) = x_{t,\theta}^{n+1}$ where $x_{t,\theta}^{n+1}$ is a solution to the following SDE:
\begin{align} \label{mapeqn}
dx_{t,\theta}^{n+1} & = \biggl( b(x_{t,\theta}^n,\mu_t) +  \varphi_\theta (t,x_{t,\theta}^{n+1}|\{\mu_t\}) \biggl) dt + \sigma  dW_t , \quad x_{s,\theta}^{n+1} = x.
\end{align}
By~\cite{Z1974}, for any given $x_{t,\theta}^n$, the SDE (\ref{mapeqn}) has a unique solution $x_{t,\theta}^{n+1}$, so the mapping $\Phi$ is well defined. Then, for any $n\in \mathbb{N}$, 
\begin{align*}
d(x_{t,\theta}^{n+1}-x_{t,\theta}^{n+2}) = \biggl( b(x_{t,\theta}^n,\mu_t)-b(x_{t,\theta}^{n+1} ,\mu_t) + \varphi_\theta (t,x_{t,\theta}^{n+1 }|\{\mu_t\}) - \varphi_\theta (t,x_{t,\theta}^{n+2} |\{\mu_t\}) \biggl) dt. 
\end{align*}
Because $\varphi_\theta (t,x|\{\mu_t\}) $ is nonincreasing in $x$, 
\begin{align*}
&d(x_{t,\theta}^{n+1}-x_{t,\theta}^{n+2})^2
\\&  =2(x_{t,\theta}^{n+1}-x_{t,\theta}^{n+2})\biggl( b(x_{t,\theta}^n,\mu_t)-b(x_{t,\theta}^{n+1} ,\mu_t) + \varphi_\theta (t,x_{t,\theta}^{n+1}|\{\mu_t\}) - \varphi_\theta (t,x_{t,\theta}^{n+2} |\{\mu_t\}) \biggl) dt
\\&  \leq 2 Lip(b) |x_{t,\theta}^{n+1}-x_{t,\theta}^{n+2}|  | x_{t,\theta}^{n}-x_{t,\theta}^{n+1}| dt
\\&  \leq  Lip(b) \biggl( |x_{t,\theta}^{n+1}-x_{t,\theta}^{n+2}|^2+ | x_{t,\theta}^{n}-x_{t,\theta}^{n+1}|^2 \biggl) dt.
\end{align*}
By Gronwall's inequality, for any $t\in [0,T]$,
\begin{align*}
|x_{t,\theta}^{n+1}-x_{t,\theta}^{n+2}|^2 
\le Lip(b)\exp\biggl( Lip(b)t\biggr) \int_0^t | x_{s,\theta}^{n}-x_{s,\theta}^{n+1}|^2   ds. 
\end{align*}
Hence, for any $n\in\mathbb{N}$,
\begin{align*}
|x_{t,\theta}^{n+1}-x_{t,\theta}^{n+2}|^2 
\le \frac{\biggl(Lip(b)t\biggr)^n \exp\biggl( nLip(b)t\biggr)}{n!}  | x_{t,\theta}^{1}-x_{t,\theta}^{2}|^2 . 
\end{align*}
As $n \rightarrow \infty$, $\Phi$ is a contraction mapping, and the SDE (\ref{mapeqn}) has a unique fixed point solution. Therefore, there exists a unique optimally controlled state process $x_{t,\theta}$ to problem (\ref{Control}). Furthermore, the optimal Markovian control $(\xi_{\cdot,\theta}^+,\xi_{\cdot,\theta}^- )$ to (\ref{Control}) also uniquely exists. 
 \end{proof}
 
 Next, we establish the regularity of the value function to problem (\ref{Control}). 
\begin{proposition}\label{strictconvex}
Assume \emph{(A1)--(A4)}. For any fixed $t \in [0,T]$, the value function  $v_{\theta}  (t, x|\{\mu_t\}) $ 
for problem (\ref{Control}) is strictly convex in $x$.
\end{proposition}
\begin{proof}
Fix any $x_1,x_2\in \mathbb{R}$ and any $\lambda \in [0,1]$.  
For any $ (\xi_\cdot^{1,+},  \xi_\cdot^{1,-}) \in \mathcal{U}_{\theta}  $ and $ (\xi_\cdot^{2,+},  \xi_\cdot^{2,-}) \in \mathcal{U}_{\theta} $, by the convexity of $f$, 
\begin{align*}
 & \lambda J_{\theta} (s,x_1 , \xi_\cdot^{1,+}, \xi_\cdot^{1,-}|\{\mu_t\}) + (1-\lambda) {J_{\theta}}  (s,x_2 , \xi_\cdot^{2,+}, \xi_\cdot^{2,-}| \{\mu_t\})
 \\ \geq & J_{\theta} (s,\lambda x_1 + (1-\lambda) x_2 ,\lambda \xi_\cdot^{1,+} + (1-\lambda) \xi_\cdot^{1,+},\lambda\xi_\cdot^{2,+} + (1-\lambda)\xi_\cdot^{2,-}| \{\mu_t\})
 \\\geq & v_{\theta}  (s, \lambda x_1 + (1-\lambda) x_2| \{\mu_t\}). 
\end{align*} 
Since this holds for any $ (\xi_\cdot^{1,+},  \xi_\cdot^{1,-}) \in \mathcal{U}_{\theta} $ and $ (\xi_\cdot^{2,+},  \xi_\cdot^{2,-}) \in \mathcal{U}_{\theta}  $, 
 \begin{align*}
  \lambda  v_{\theta} (s,x_1|\{\mu_t\})  + (1-\lambda) J_{\theta}  (s,x_2 , \xi_\cdot^{2,+}, \xi_\cdot^{2,-}|\{\mu_t\}) 
 \geq   v_{\theta}  (s, \lambda x_1 + (1-\lambda) x_2|\{\mu_t\}),
\end{align*}  
 \begin{align*}
  \lambda  v_{\theta} (s,x_1|\{\mu_t\})  + (1-\lambda)  v_{\theta} (s,x_2| \{\mu_t\})
 \geq   v_{\theta}  (s, \lambda x_1 + (1-\lambda) x_2|\{\mu_t\}). 
\end{align*} 
Hence, $v_{\theta}  (s, x|\{\mu_t\})$ is convex in $x$.
By Proposition \ref{optimization}, $v_{\theta}  (s, x| \{\mu_t\})$ is a $\mathcal{C}^{1,2} ([0,T]\times \mathbb{R})$ solution to the equation 
\begin{align*}
- \partial_t  v_{\theta}  =\min \biggl\lbrace  ( \partial_x v_{\theta} + \gamma_1)\theta,(- \partial_x v_{\theta} + \gamma_2)\theta, 0	\biggl\rbrace  +b(x,\mu )  \partial_x v_{\theta} +f(x ,\mu )+ \frac{\sigma^2}{2}\partial_{xx} v_{\theta}.
\end{align*}
Since $f(x,\mu)$ is not linear in $x$, the solution to this equation is also nonlinear in $x$. Hence, $v_{\theta}  (s, x|\{\mu_t\})$  is strictly convex.  
\end{proof}

With this convexity, we have 
\begin{theorem} \label{thetainfty}
Assume \emph{(A1)--(A4)}. Then for any $(s,x) \in [0,T]\times \mathbb{R}$, as $\theta\rightarrow \infty$, the value function $v_{\theta} (s,x|\{ \mu_t\}) $ of  \emph{(\ref{Control})} converges to the value function $v(s,x|\{ \mu_t\}) $ of \emph{(\ref{Control-FV})}. Moreover, there exists an optimal control of a feedback form for \emph{(\ref{Control-FV})}.
\end{theorem}
\begin{proof}Fix $\{\mu_t\} \in \mathcal{M}_{[0,T]}$. 
For any $(\zeta_{\cdot}^+,\zeta_{\cdot}^-) \in \mathcal{U}$, since each path of  a finite variation process is almost everywhere differentiable, there exists a sequence of  bounded velocity functions which converges to the path as $\theta \rightarrow \infty$. Hence, there exists a sequence  $\{(\zeta_{\cdot,\theta}^+ ,\zeta_{\cdot,\theta}^-)\}_{\theta \in [0,\infty)}$ such that  $(\zeta_{\cdot,\theta}^+,\zeta_{\cdot,\theta}^- ) \in \mathcal{U}_{\theta}$ and   $\mathbb{E} \int_0^T  |\dot{\zeta}_{t,\theta}^+ dt - d\zeta_{t}^+ | \rightarrow 0 ,  \mathbb{E} \int_0^T  |\dot{\zeta}_{t,\theta}^- dt - d\zeta_{t}^-  |   \rightarrow 0$ as $\theta \rightarrow \infty$.

Define $\epsilon_\theta$ as
\begin{align}\label{epsilontheta}
\epsilon_\theta =  O\biggl(  \mathbb{E}\int_0^T |\dot{\zeta}_{t,\theta}^+dt  -   d\zeta_{t}^{+} |+ \mathbb{E}\int_0^T  |\dot{\zeta}_{t,\theta}^- dt-d\zeta_{t}^{-}|   \biggl),
\end{align}
and $\epsilon_\theta \rightarrow 0$ as $\theta \rightarrow \infty$.  

Denote 
\begin{align*}
d\hat{x}_{t,\theta} & = (b(\hat{x}_{t,\theta}, \mu_t ) +\dot{\zeta}_{t,\theta}^+ - \dot{\zeta}_{t,\theta}^-)  dt + \sigma dW_t, \quad \hat{x}_{s,\theta}= x, \text{ and }  
\\   d\hat{x}_t  & = b(\hat{x}_t,\mu_t) dt+ \sigma  dW_t + d\zeta_{t}^{+}- d\zeta_{t}^{-}  , \quad \quad \hat{x}_{s-} = x.
\end{align*}
Then, for any $\tau \in [s,T]$, 
\begin{align*}
|\hat{x}_{\tau,\theta} - \hat{x}_\tau| &\le \int_s^\tau | b(\hat{x}_{t,\theta},\mu_t) - b(\hat{x}_t,\mu_t)| dt + \int_s^\tau |\dot{\zeta}_{t,\theta}^+dt  -   d\zeta_{t}^{+} |+ \int_s^\tau  |\dot{\zeta}_{t,\theta}^- dt - d\zeta_{t}^{-}|
\\  & \le   \int_s^\tau Lip(b)|\hat{x}_{t,\theta} - \hat{x}_t| dt + \int_s^\tau |\dot{\zeta}_{t,\theta}^+dt  -   d\zeta_{t}^{+} |+ \int_s^\tau  |\dot{\zeta}_{t,\theta}^- dt - d\zeta_{t}^{-}| .
\end{align*}
By  Gronwall's inequality, 
\begin{align*}
\mathbb{E} |\hat{x}_{\tau,\theta} - \hat{x}_\tau| \le O\left(\mathbb{E}\int_0^\tau |\dot{\zeta}_{t,\theta}^+dt  -   d\zeta_{t}^{+} |+ \mathbb{E}\int_0^\tau  |\dot{\zeta}_{t,\theta}^- dt - d\zeta_{t}^{-}| \right).
\end{align*}
Consequently, 
\begin{align*}
& \biggl|J_(s,x,\zeta_{t}^+,\zeta_{t}^- |\{\mu_t\} ) - J_{\theta}  (s,x,\zeta_{t,\theta}^+,\zeta_{t,\theta}^-|\{\mu_t\} )\biggl| 
 \\  \le & \ \mathbb{E}\biggl[ \biggl| \int_s^T f(\hat{x}_t, \mu_t) - f(\hat{x}_{t,\theta},\mu_t) +\gamma_1 d\zeta_{t}^+ + \gamma_2  d\zeta_{t}^- - \gamma_1  \dot{\zeta}_{t,\theta}^+dt -  \gamma_2 \dot{\zeta}_{t,\theta}^-dt \biggl| \biggl]
  \\  \le & \ \mathbb{E}\biggl[ \int_s^T  Lip(f)|\hat{x}_t -\hat{x}_{t,\theta}|     + \gamma_1 |d\zeta_{t}^+ -\dot{\zeta}_{t,\theta}^+ dt| +  \gamma_2 |d\zeta_{t}^- -  \dot{\zeta}_{t,\theta}^-dt |  \biggl]
  \\\le & \ O\biggl(\mathbb{E}\int_0^T |\dot{\zeta}_{t,\theta}^+dt  -   d\zeta_{t}^{+} |+ \mathbb{E}\int_0^T  |\dot{\zeta}_{t,\theta}^- dt-d\zeta_{t}^{-}| \biggl).
\end{align*}
Therefore, 
$\biggl |v (s,x|\{ \mu_t\}) -  v_{\theta} (s,x|\{ \mu_t\})\biggl|  \rightarrow 0  \text{ as } \theta \rightarrow 0$.

Now a similar argument as in Corollary (4.11) \cite{MT1989} shows the existence of a feedback control for 
\emph{(\ref{Control-FV})}.
\end{proof}

\subsection{Game (MFG-BD) }\label{proof} 
Our next step is to analyze the game (MFG-BD).  In particular, we see that 
\begin{theorem}
Assume \emph{(A1)--(A6)}. Then there exists a unique solution $((\xi_\cdot^{*+},\xi_\cdot^{*-}),\{\mu_t^*\} )$ of \emph{(\ref{MFGbounded1})}. Moreover, the corresponding value function 
$v_{\theta}(s,x)$ for \emph{(\ref{MFGbounded1})} is in {$C^{1,2}( [0,T) \times \mathbb{R})\bigcap C( [0,T] \times \mathbb{R})$} with  a polynomial growth.
\label{mainthm}
\end{theorem}

The proof of the existence of the MFG solution proceeds as follows. 

First, from Proposition \ref{optimization} we see that for any given fixed $\{\mu_t\}$ there exists a unique optimal control function as $\varphi_\theta(t,x |\{\mu_t\} ) $. Now, one  can define a mapping $\Gamma_1 $ from
$\mathcal{M}_{[0,T]}$ to a class of pairs of the optimal control function $\varphi_{\theta}$ and the fixed flow of probability measures $\{\mu_t\}$ such that  $$\Gamma_1 (\{\mu_t\}) = \biggl( \varphi_\theta(t,x|\{\mu_t\}) , \{\mu_t\}\biggl).$$
Moreover, by Proposition \ref{optimization} the optimally controlled process  $x_{t,\theta} $ under the fixed $\{\mu_t\}$ exists uniquely with 
\begin{align*} 
d x_{t,\theta}  = \biggl( b(x_{t,\theta},\mu_t) +  \varphi_\theta(t,x_{t,\theta}|\{\mu_t\}) \biggl)  dt + \sigma dW_t, \quad \quad x_{s,\theta} = x.
\end{align*} 
    Consequently, we can define $\Gamma_2 $  so that  $$\Gamma_2 \biggl( \varphi_\theta(t,x|\{\mu_t\}), \{\mu_t\}\biggl) = \{ \tilde{\mu}_t \} ,$$ where $ \tilde{\mu}_t $ is the probability measure of $x_{t,\theta}$ for each $t\in [0,T]$. 

Now,  define a mapping $\Gamma$ as $$\Gamma(\{ \mu_t\})= \Gamma_2 \circ \Gamma_1 (\{\mu_t\}) = \{ \tilde{\mu}_t\}.$$ We will use the Schauder fixed point theorem~\cite[Theorem 4.1.1]{Smart1980} to show the existence of a fixed point. 
 The key is to prove that $\Gamma$ is a continuous mapping of $\mathcal{M}_{[0,T]} $ into $  \mathcal{M}_{[0,T]}$, and the range of $\Gamma$ is relatively compact \cite{B2013}.

\begin{proposition}\label{MM}
 Assume \emph{(A1)--(A4)}.
$\Gamma$ is a mapping from $\mathcal{M}_{[0,T]}$ to $\mathcal{M}_{[0,T]}$.
\end{proposition}    
\begin{proof}
  For any $\{\mu_t\}$  in $\mathcal{M}_{[0,T]}$, let us prove that $\{\tilde{\mu}_t\} = \Gamma (\{\mu_t\})$ is also in $\mathcal{M}_{[0,T]}$.  Without loss of generality, suppose $s  > t$, and $$x_{s} = x_t + \int_t^{s} \biggl(b(x_r,\mu_r)+ \varphi_\theta(r,x_r|\{\mu_t\} )\biggl) dr + \int_t^{s}\sigma dW_r.$$
Since $b(x, \mu )  $ is  bounded, $|\varphi_\theta(s,x_s|\{\mu_t\} )| \leq \theta$,  and $\mathbb{E}\biggl| (b(x_r,\mu_r)+ \varphi_\theta(r,x_r|\{\mu_t\} ))\biggl|  \le M $ for large $M$ and for any $r \in [0,T]$,
\begin{align*}
D^1(\tilde{\mu}_s,\tilde{\mu}_t ) &\leq \mathbb{E} | x_s -x_t   | 
\\ & \leq \mathbb{E} \int_t^s \biggl|b(x_r,\mu_r)+ \varphi(r,x_r |\{\mu_t\})\biggl| dr  + \sigma \mathbb{E}  \sup_{r \in [t,s]}\limits |W_r-W_t | 
\\ & \leq   M|s-t|  + \sigma \mathbb{E}  \sup_{r \in [t,s]}\limits |W_r-W_t | \leq  M|s-t| + \sigma |s-t|^{\frac{1}{2}}.   
\end{align*}
  Therefore,
 $ \sup_{s\neq t}\frac{ D^1(\tilde{\mu}_t,\tilde{\mu}_s) }{|t-s|^{\frac{1}{2}}} \leq c$.
For any $t \in [0,T]$, since  $|b(x,\mu)|$ is bounded, 
\begin{align*}
\int_\mathbb{R} |x|^2 \tilde{\mu}_t(dx) \leq 2 \mathbb{E} \biggl[ \int_\mathbb{R} |x|^2 d \tilde{\mu}_0 +  c_2^2 t^2 + \sigma^2 t \biggl]  \leq 2 \mathbb{E} \biggl[ \int_\mathbb{R} |x|^2 d \tilde{\mu}_0 +  c_1^2 T^2 + \sigma^2 T \biggl],
\end{align*} 
 and  $\sup_{t \in [0,T]} \limits \int_\mathbb{R} |x|^2 \tilde{\mu}_t(dx) \leq c$.
 \end{proof}

\begin{proposition} \label{continuous}
Assume \emph{(A1)--(A6)}. $\Gamma : \mathcal{M}_{[0,T]} \rightarrow \mathcal{M}_{[0,T]}$ is continuous.
\end{proposition}
\begin{proof}
Let $ \{ \mu_t^n \} \in \mathcal{M}_{[0,T]}$ for $n = 1, \ldots,$ be a sequence of flows of probability measures $d_\mathcal{M}(\{\mu_t^n\},\{\mu_t\}) \rightarrow 0 $ as  $n \rightarrow \infty$, for some $\{\mu_t\} \in \mathcal{M}_{[0,T]}$.
Fix $\tau \in [0,T)$. By Proposition \ref{optimization}, 
for each $\{\mu_t^n\} $, problem (\ref{Control}) has a value function $v_{\theta}^n(s,x | \{\mu_t^n\})$ with the optimal control $\varphi^n(t,x)$.  
Let $\{x_t^n\}$ be the  corresponding optimal controlled process:
$$dx_t^n =  \biggl(b(x_t^n ,\mu_t^n)+\varphi^n_\theta(t,x_t^n | \{ \mu_t^n \} )\biggl)dt + \sigma dW_t, \quad \tau \leq t \leq T, \quad x_\tau ^n =x.$$
Let $ \{\tilde{\mu}_t^n\}$ be a flow of  probability measures of $\{x_t^n\}$, then $\Gamma (\{\mu_t^n\} ) = \{\tilde{\mu}_t^n\}$. 

Similarly, for each $\{\mu_t \} $, problem (\ref{Control}) has a   value function $v_{\theta} (s,x|b\{\mu_t \})$ with the optimal control $\varphi_\theta(t,x | \{ \mu_t \})$. 
Let $\{x_t\}$ be the  corresponding optimal controlled process:
$$dx_t =  \biggl(b(x_t ,\mu_t)+ \varphi_\theta(t,x_t | \{ \mu_t \})\biggl)dt + \sigma dW_t, \quad \tau \leq t \leq T, \quad x_\tau =x.$$
Let $ \{\tilde{\mu}_t\}$ be a flow of  probability measures of $\{x_t\}$, then $\Gamma (\{\mu_t\} ) = \{\tilde{\mu}_t\}$. 

To show that $\Gamma$ is continuous, 
 we need to show $$d_{\mathcal{M}} \biggl(\{\tilde{\mu}_t^n \}, \{\tilde{\mu}_t\}\biggl) \rightarrow 0 \text{ as  } n\rightarrow \infty.$$
  
This is established in four steps.

Step 1. We first establish some relation between $D^2 (\{\tilde{\mu}_t^n \}, \{\tilde{\mu}_t\})$ and $D^2(\{\mu_t^n \}, \{\mu_t\})$.
Note here $ D^1 (\tilde{\mu}_t,\tilde{\mu}_t^n) \le D^2(\tilde{\mu}_t,\tilde{\mu}_t^n) $.

For any $s \in [\tau ,T]$,  
\begin{align*}
d(x_s -x_s^n)=  \biggl(b(x_s,\mu_s)-b(x_s^n,\mu_s^n) + \varphi_\theta(s, x_s | \{ \mu_t \})-  \varphi^n_\theta(s,x^n_s | \{ \mu_t^n \})\biggl) ds.
\end{align*}
Then, for any $t \in [\tau ,T]$,
\begin{align*}
 |x_t-x^n_t|^2 & = 2 \int_\tau^t \biggl(b(x_s ,\mu_s )-b(x^n_s,  \mu_s^n)+ \varphi_\theta(s,x_s | \{ \mu_t \} )-  \varphi^n_\theta  (s,x_s^n | \{ \mu^n_t \}) \biggl)(x_s-x_s^n)ds
\\& \leq 2\int_\tau^t Lip(b) \biggl(|x_s-x_s^n|+D^1(\mu_s,\mu_s^n)\biggl)|x_s-x_s^n|
\\& \quad +\biggl( \varphi_\theta(s,x_s | \{ \mu_t \})- \varphi^n_\theta (s,x_s^n | \{ \mu^n_t \})\biggl)(x_s -x_s^n)ds.
\end{align*} 
\begin{align*}
& Lip(b)\biggl(|x_s-x_s^n|+D^1(\mu_s,\mu_s^n)\biggl)|x_s-x_s^n|
   \leq      Lip(b_0)|x_s - x_s^n|^2 + \frac{Lip(b)}{2}\biggl((D^1(\mu_s,\mu_s^n))^2+  |x_s-x_s^n|^2\biggl).
\end{align*}
By  Assumption (A6), 
\begin{align*}
& ( \varphi_\theta(s,x_s | \{ \mu_t \})- \varphi^n_\theta (s,x_s^n | \{ \mu^n_t \}))(x_s -x_s^n)
\\ \leq & \ \biggl( \varphi_\theta(s,x_s | \{ \mu_t \} )- \varphi^n_\theta(s,x_s | \{ \mu^n_t \} )+ \varphi^n_\theta(s,x_s | \{ \mu^n_t \}  )-  \varphi^n_\theta (s,x_s^n | \{ \mu^n_t \})\biggl)(x_s -x_s^n)
\\ \leq & \ ( \varphi_\theta(s,x_s | \{ \mu_t \} )- \varphi^n_\theta(s,x_s | \{ \mu^n_t \} ) )(x_s -x_s^n)
\\ \leq & \ \frac{1}{2}\biggl(|\varphi_\theta(s,x_s | \{ \mu_t \} )- \varphi^n_\theta(s,x_s | \{ \mu^n_t \} )|^2 +|x_s -x_s^n|^2\biggl).
\end{align*}

Consequently, 
\begin{align*}
 |x_t-x_t^n|^2  \leq &  \int_\tau^t  ( 3Lip(b )+1) |x_s-x_s^n|^2    +  Lip(b_0) (D^1(\mu_s,\mu_s^n))^2   +  |\varphi_\theta(s,x_s | \{ \mu_t \} )- \varphi^n_\theta(s,x_s | \{ \mu^n_t \} )|^2  ds . 
\end{align*}
 By  Gronwall's inequality,  
\begin{align}\label{ineqcontinuity1}
 (D^2(\tilde{\mu}_t,\tilde{\mu}_t^n))^2   &\leq   c_2  \int_\tau^t Lip(b ) (D^1(\mu_s,\mu_s^n))^2   +  \mathbb{E}\biggl[|\varphi_\theta(s,x_s | \{ \mu_t \} )- \varphi^n_\theta(s,x_s | \{ \mu^n_t \} )|^2  \biggl] ds , 
\end{align}
for some constant  $c_2$ depending on $T$ and $Lip(b )$.

Step 2. Now we prove that for any $(t,x)\in [\tau,T]\times \mathbb{R} $,   $$\partial_x v_{\theta}^n(t,x|\{\mu_t^n\})\rightarrow \partial_x v(t,x|\{\mu_t \})\text{ as } n\rightarrow \infty.$$ 
By Proposition \ref{optimization}, $v_{\theta}$ and $v_{\theta}^n$ are the solutions to the HJB Eqn. (\ref{HJBHJBHJB}).
 For notation simplicity, let us denote  
$$\varphi_{1, \theta}(s,x| \{ \mu_t \}) = \max \{\varphi_\theta(s,x| \{ \mu_t \}),0\}, \ \  
\varphi_{2,\theta}(s,x| \{ \mu_t \}) =- \min \{\varphi_\theta(s,x| \{ \mu_t \}),0\},$$ 
 $$\varphi^n_{1, \theta}(s,x | \{ \mu^n_t \}) = \max \{\varphi^n_\theta(s,x | \{ \mu^n_t \}),0\}, \ \ \varphi^n_{2, \theta}(s,x | \{ \mu^n_t \}) =- \min \{\varphi^n_\theta(s,x | \{ \mu^n_t \}),0\}.$$ 
Since $\varphi_{1, \theta| \{ \mu_t \}},\varphi_{2, \theta| \{ \mu_t \}}$ are optimal controls, using It\^o's formula and the HJB Eqn. (\ref{HJBHJBHJB}), we obtain

\begin{align}\label{eqeqeq}
\begin{split}
  & -v_{\theta}(\tau,x|\{\mu_t\} ) \\&= v_{\theta}(T,x_T|\{\mu_t  \}) -  v_{\theta}(\tau,x|\{\mu_t\} ) 
\\&  = -  \int_\tau^T \biggl( f(x_s,\mu_s) + \gamma_1 \varphi_{1, \theta} (s,x_s| \{ \mu_t \})+ \gamma_2 \varphi_{2, \theta} (s,x_s| \{ \mu_t \})\biggl) ds   + \int_\tau^T \sigma \partial_x v_{\theta} (s,x_s|\{\mu_t \}) dW_s.
\end{split}
\end{align}
 
Similarly, for any $n \in \mathbb{N}$, applying It\^o's formula to $v_{\theta}^n(s,x)$ and $\{x_t\}$  yields
\begin{align*}
 &v_{\theta}^n(T,x_T|\{\mu_t^n\}) - v_{\theta}^n(\tau,x|\{\mu_t^n\})  
 \\ = &\int_\tau^T \partial_t v^n_{\theta} (s,x_s|\{\mu_t^n\}) +  ( b(x_s,\mu_s) + \varphi_\theta(s,x_s| \{ \mu_t \}) ) \partial_x v_{\theta}^n(s,x_s|\{\mu_t^n\}) + \frac{\sigma^2}{2} \partial_{xx} v^n_{\theta}(s,x_s|\{\mu_t^n\}) ds 
  \\&  + \int_\tau^T\sigma \partial_x v^n_{\theta}(s,x_s|\{\mu_t^n\}) dW_s
  \\  =& \int_\tau^T \partial_t v^n_{\theta} (s,x_s|\{\mu_t^n\}) +  ( b(x_s,\mu_s^n) +  \varphi^n_\theta(s,x_s | \{ \mu^n_t \} ) ) \partial_x v^n_{\theta}(s,x_s|\{\mu_t^n\}) + \frac{\sigma^2}{2} \partial_{xx} v^n_{\theta}(s,x_s|\{\mu_t^n\}) ds  
  \\& + \int_\tau^T\sigma \partial_x v^n_{\theta}(s,x_s|\{\mu_t^n\}) dW_s
  \\& - \int_\tau^T(b(x_s,\mu_s^n) -b(x_s,\mu_s) +\varphi^n_\theta(s,x_s | \{ \mu^n_t \}) -\varphi_\theta (s,x_s| \{ \mu_t \}))   \partial_x v^n_{\theta}(s,x_s|\{\mu_t^n\})  ds
\\ =& - \int_\tau^T \biggl( f(x_s,\mu_s^n) + \gamma_1 \varphi_{1, \theta}^n (s,x_s | \{ \mu^n_t \})+ \gamma_2 \varphi_{2, \theta}^n (s,x_s | \{ \mu^n_t \})\biggl) ds + \int_\tau^T\sigma \partial_x v^n_{\theta}(s,x_s|\{\mu_t^n\}) dW_s 
\\&  - \int_\tau^T(b(x_s,\mu_s^n) -b(x_s,\mu_s) + \varphi^n_\theta(s,x_s | \{ \mu^n_t \}) -\varphi_\theta (s,x_s| \{ \mu_t \}))   \partial_x v^n_{\theta}(s,x_s|\{\mu_t^n\})  ds.
\end{align*}
The last equality is due to the HJB Eqn. (\ref{HJBHJBHJB}). Hence, 
\begin{align}\label{eqnn}
\begin{split}
  v_{\theta}^n(\tau,x|\{\mu_t^n\} ) & =   \int_\tau^T \biggl( f(x_s ,\mu_s^n ) + \gamma_1 \varphi_{1, \theta}^n (s,x_s | \{ \mu^n_t \} )+ \gamma_2 \varphi_{2, \theta}^n (s,x_s | \{ \mu^n_t \} )  \biggl) ds -  \int_\tau^T \sigma \partial_x v^n_{\theta} (s,x_s|\{\mu_t^n\} ) dW_s \\& + \int_\tau^T \biggl(b(x_s,\mu_s^n) -b(x_s,\mu_s) + \varphi^n_\theta(s,x_s | \{ \mu^n_t \}) -\varphi_\theta (s,x_s| \{ \mu_t \}))   \partial_x v^n_{\theta}(s,x_s|\{\mu_t^n\}\biggl)  ds.
\end{split}
\end{align} 
Denote $
H(s,x ) = \inf_{\dot{\xi}^+,\dot{\xi}^- \in [0,\theta]} \limits \{ ( \dot{\xi}^+-\dot{\xi}^-)\partial_x  v_{\theta}(s,x|\{\mu^n_t\}) +  \gamma_1 \dot{\xi}^++ \gamma_2 \dot{\xi}^-\}, $ and\\ $H^n(s,x ) = \inf_{\dot{\xi}^+,\dot{\xi}^- \in [0,\theta]} \limits \{ ( \dot{\xi}^+-\dot{\xi}^-)\partial_x v^n_{\theta} (s,x|\{\mu_t^n\}) + \gamma_1 \dot{\xi}^++ \gamma_2 \dot{\xi}^-\}. $ 
Then for any $\dot{\xi}^+,\dot{\xi}^- \in [0,\theta]$, 
\begin{align*}
&\left| \left(( \dot{\xi}^+-\dot{\xi}^-)\partial_x v_{\theta}(s,x|\{\mu_t \} ) +  \gamma_1 \dot{\xi}^++ \gamma_2 \dot{\xi}^- \right)- \left( (\dot{\xi}^+-\dot{\xi}^-)\partial_x v^n_{\theta}(s,x|\{\mu_t^n\}) + \gamma_1 \dot{\xi}^++ \gamma_2 \dot{\xi}^-\right) \right|
\\ &\leq \biggl|   \dot{\xi}^+ \biggl( \partial_x v_{\theta}(s,x|\{\mu_t \})-{\partial_x v^n_{\theta}(s,x|\{\mu_t^n\})\biggl)} -  \dot{\xi}^-\biggl( \partial_x v_{\theta}(s,x|\{\mu_t \})-{\partial_x v^n_{\theta}(s,x|\{\mu_t^n\}) \biggl)} \biggl|
\\& \leq 2\theta \left|     \partial_x v_{\theta}(s,x|\{\mu_t \})-\partial_x v^n_{\theta} (s,x|\{\mu_t^n\})   \right|.
\end{align*}
Hence, for any $s,x \in [\tau,T]\times \mathbb{R}$,
$$ |H(s,x)-H^n(s,x)|  \leq 2\theta \biggl|     \partial_x v_{\theta}(s,x|\{\mu_t \})-\partial_x v^n_{\theta}(s,x|\{\mu_t^n\})\biggl|.$$
By definition, 
\begin{align*}
 2\theta &\biggl|     \partial_x v_{\theta}(s,x|\{\mu_t \})-\partial_x v^n_{\theta}(s,x|\{\mu_t^n\}) \biggl| 
\\\geq  & \biggl| \biggl( \varphi_{1, \theta}(t,x| \{ \mu_t \}) -\varphi_{2, \theta}(t,x| \{ \mu_t \}) \biggl)\partial_x v_{\theta}(s,x|\{\mu_t \}) + \gamma_1 \varphi_{1, \theta}(t,x| \{ \mu_t \})+ \gamma_2 \varphi_{2, \theta}(t,x| \{ \mu_t \})
\\&  - \biggl( \varphi_{1, \theta}^n(t,x | \{ \mu^n_t \}) -\varphi_{2, \theta}^n(t,x | \{ \mu^n_t \}) \biggl)\partial_x v^n_{\theta}(s,x|\{\mu_t^n\} ) + \gamma_1 \varphi_{1, \theta}^n(t,x | \{ \mu^n_t \})+\gamma_2 \varphi_{2, \theta}^n(t,x | \{ \mu^n_t \}) \biggr|
\\ = & \biggl| \biggl(\gamma_1+\partial_x v_{\theta}(s,x|\{\mu_t \}) \biggl) \biggl(\varphi_{1, \theta}(t,x| \{ \mu_t \})- \varphi_{1, \theta}^n(t,x | \{ \mu^n_t \})\biggl) \\
+ & \biggl(\gamma_2-\partial_x v_{\theta}(s,x|\{\mu_t^n\}) \biggl) \biggl(\varphi_{2, \theta}(t,x| \{ \mu_t \})- \varphi_{2, \theta}^n(t,x | \{ \mu^n_t \})\biggl) 
\\+& \biggl(\partial_x v_{\theta} (s,x|\{\mu_t \})-\partial_x v^n_{\theta} (s,x|\{\mu_t^n\})\biggl)\biggl( \varphi_{1, \theta}^n(t,x | \{ \mu^n_t \}) -\varphi_{2, \theta}^n(t,x | \{ \mu^n_t \})\biggl ) \biggr|
\\\geq &  \biggl| \biggl( \gamma_1+\partial_x v_{\theta}(s,x|\{\mu_t \}) \biggl) \biggl(\varphi_{1, \theta}(t,x| \{ \mu_t \})- \varphi_{1, \theta}^n(t,x | \{ \mu^n_t \})\biggl) \\
+ &\biggl( \gamma_2 -\partial_x v_{\theta}(s,x|\{\mu_t \}) \biggl) \biggl(\varphi_{2, \theta}(t,x| \{ \mu_t \})- \varphi_{2, \theta}^n(t,x | \{ \mu^n_t \})\biggl) \biggr| 
 - \theta \biggl|  \partial_x v_{\theta} (s,x|\{\mu_t \})-\partial_x v^n_{\theta} (s,x |\{\mu_t^n\}) \biggr|.
\end{align*}
Hence, 
\begin{align} \label{eqeqeqeq}
\begin{split}
 3\theta  \biggl|     \partial_x v_{\theta}(s,x|\{\mu_t \})-\partial_x v^n_{\theta}(s,x|\{\mu_t^n\})   \biggl| 
 & \geq \biggl| \biggl( \gamma_1 +\partial_x v_{\theta}(s,x|\{\mu_t \}) \biggl) \biggl(\varphi_{1, \theta}(s,x| \{ \mu_t \})- \varphi_{1, \theta}^n(s,x | \{ \mu^n_t \})\biggl)
 \\&  + \biggl( \gamma_2 -\partial_x v_{\theta}(s,x|\{\mu_t \}) \biggl) \biggl(\varphi_{2, \theta}(s,x| \{ \mu_t \})- \varphi_{2, \theta}^n(s,x | \{ \mu^n_t \})\biggl) \biggr|. 
\end{split}
\end{align} 
Similarly,
\begin{align} \label{eqeqeqeqeq}
\begin{split}
  3\theta \biggl|     \partial_x v_{\theta}(s,x|\{\mu_t\})-\partial_x v^n_{\theta}(s,x|\{\mu_t^n\}) \biggl| 
  & \geq  \biggl| \biggl( \gamma_1 +\partial_x v^n_{\theta}(s,x|\{\mu_t^n\}) \biggl) \biggl(\varphi_{1, \theta}(s,x| \{ \mu_t \})- \varphi_{1, \theta}^n(s,x | \{ \mu^n_t \})\biggl)
 \\&  + \biggl( \gamma_2 -\partial_x v^n_{\theta}(s,x|\{\mu_t^n\}) \biggl) \biggl(\varphi_{2, \theta}(s,x| \{ \mu_t \})- \varphi_{2, \theta}^n(s,x | \{ \mu^n_t \})\biggl) \biggr|. 
\end{split}
\end{align}

Step 3. We can further show $\varphi^n_\theta( s,x | \{ \mu^n_t \}) \rightarrow \varphi_\theta(s,x| \{ \mu_t \})$ for any $ s,x \in [0,T]\times \mathbb{R}$ as $n \rightarrow \infty$.

Indeed, from Eqns. (\ref{eqeqeq}) and (\ref{eqnn}) and by It\^o's isometry and Cauchy--Schwartz inequality, 
\begin{align*}
&\biggl( v_{\theta}(\tau,x |\{\mu_t\}) - v^n_{\theta}(\tau,x |\{\mu_t^n\})\biggl)^2 +  \sigma^2   \mathbb{E}\biggl[\int_\tau^T\biggl(\partial_x v_{\theta}(s,x_s|\{\mu_t \}) -  \partial_x v^n_{\theta}(s,x_s|\{\mu_t^n\}) \biggl)^2 ds \biggl]    
    \\ \leq   &  3(T-\tau)  \mathbb{E}  \biggl[  \int_\tau^T   \biggl(f(x_s ,\mu_s) - f(x_s ,\mu_s^n )\biggl)^2   + \biggl( (b(x_s,\mu_s) -b(x_s,\mu_s^n) ) \partial_x v^n_{\theta}(s,x_s|\{\mu_t \}) \biggl)^2   \\& +  \biggl( (\gamma_1+ \partial_x v^n_{\theta}(s,{x_s}|\{\mu_t^n\})) (\varphi_{1, \theta}  (s,x_s| \{ \mu_t \} )-\varphi_{1, \theta}^n (s,x_s | \{ \mu^n_t \} ))
    \\& + ( \gamma_2 - \partial_x v^n_{\theta}(s,{x_s}|\{\mu_t^n\} ))(\varphi_{2, \theta}(s,x_s| \{ \mu_t \} )- \varphi_{2, \theta}^n (s,x_s | \{ \mu^n_t \} ))   \biggl)^2 ds  \biggl]
  \\ \leq   &   3(T-\tau)  \mathbb{E}  \biggl [ \int_\tau^T   \biggl( Lip(f) D^1(\mu_s ,\mu_s^n) \biggl)^2  +   \biggl( Lip(b) D^1(\mu_s ,\mu_s^n) \biggl| \partial_x v^n_{\theta}(s,x_s|\{\mu_t^n\}) \biggl|    \biggl)^2 \\& +  \biggl( (\gamma_1+ \partial_x v^n_{\theta}(s,x_s |\{\mu_t^n\})) (\varphi_{1, \theta}  (s,x_s| \{ \mu_t \} )-\varphi_{1, \theta}^n (s,x_s | \{ \mu^n_t \} ))
  \\&+ (\gamma_2- \partial_x v^n_{\theta}(s,x_s |\{\mu_t^n\}))(\varphi_{2, \theta}(s,x_s| \{ \mu_t \} )- \varphi_{2, \theta}^n (s,x_s | \{ \mu^n_t \} )) \biggl) ^2      ds  \biggl]
  \\ \leq &   3(T-\tau)  \mathbb{E}  \biggl [ \int_\tau^T   \biggl( Lip(f) D^1(\mu_s ,\mu_s^n) \biggl)^2  +   \biggl( Lip(b) D^1(\mu_s ,\mu_s^n)| \partial_x v^n_{\theta}(s,x_s|\{\mu_t^n\})|   \biggl )^2 \\& + \biggl (3\theta (\partial_x v_{\theta}(s,x_s|\{\mu_t^n\}) -  \partial_x v^n_{\theta}(s,x_s|\{\mu_t^n\}))  \biggl)^2      ds  \biggl].
\end{align*}
Let $\delta = \frac{\sigma^2}{54\theta^2}$. Then, for any $\tau \in [T-\delta, T]$, 
\begin{align*}
& \biggl( v_{\theta}(\tau,x|\{\mu_t \} ) - v^n_{\theta}(\tau,x |\{\mu_t^n\}) \biggl)^2 +  \frac{\sigma^2}{2} \mathbb{E} \biggl[\int_\tau^T(\partial_x v_{\theta}(s,x_s|\{\mu_t \}) -  \partial_x v^n_{\theta}(s,x_s|\{\mu_t^n\}) )^2 ds  \biggl]    
  \\ \leq &   3(T-\tau)  \mathbb{E}   \biggl[ \int_\tau^T   \biggl( Lip(f) D^1(\mu_s ,\mu_s^n) \biggl)^2  +   \biggl( Lip(b) D^1(\mu_s ,\mu_s^n)| \partial_x v^n(s,x_s|\{\mu_t^n\})|    \biggl)^2 ds  \biggl].
\end{align*}
Hence, for any $\tau \in [T-\delta, T]$, $$ v_{\theta}(\tau,x|\{\mu_t \} ) - v^n_{\theta}(\tau,x|\{\mu_t^n\} ) \rightarrow 0,$$and $$ \mathbb{E} \biggl[\int_\tau^T \biggl(\partial_x v_{\theta}(s,x_s|\{\mu_t \}) -  \partial_x v^n_{\theta}(s,x_s|\{\mu_t^n\})  \biggl)^2 ds  \biggl] \rightarrow 0 \text{ as } n \rightarrow \infty.$$ 

Since $\delta >0$, one can repeat this process for $[T-2\delta, T-\delta]$. Proceeding recursively, one can show that for any $(t,x) \in [0,T]\times \mathbb{R}$,  $   v^n_{\theta}(t,x|\{\mu_t^n\} )  \rightarrow v_{\theta}(t,x |\{\mu_t \}), $ and $ \mathbb{E} \biggl[\int_0^T \biggl(\partial_x v_{\theta}(s,x_s|\{\mu_t\}) -  \partial_x v^n_{\theta}(s,x_s|\{\mu_t^n\}) \biggl )^2 ds  \biggl] \rightarrow 0 \text{ as } n \rightarrow \infty.$ Hence, for any $(s,x) \in [0,T]\times \mathbb{R}$, 
 $$\partial_x v^n_{\theta} (s,x|\{\mu_t^n\}) \rightarrow \partial_x v_{\theta}(s,x|\{\mu_t\})\text{ as } n\rightarrow \infty.$$ 
 By Proposition \ref{strictconvex}, $\partial_x v^n_{\theta } (s,x|\{\mu_t^n\}),\partial_x v_{\theta}(s,x|\{\mu_t \}) $ are strictly increasing in $x$, and by definition of $\varphi^n_\theta$ and $\varphi_\theta$,  $\varphi^n_\theta( s,x | \{ \mu^n_t \}) $ converges to $\varphi_\theta(s,x| \{ \mu_t \})$ for any $(s,x) \in [0,T]\times \mathbb{R} $. 

Step 4. We are now ready to show $d_\mathcal{M} \biggl(\{\tilde{\mu}_t\},\{\tilde{\mu}_t^n\} \biggl)  \rightarrow 0$ as $n \rightarrow \infty$.
 
From previous steps, $\varphi^n_\theta( s,x_s | \{ \mu^n_t \})\rightarrow\varphi_\theta(s,x_s| \{ \mu_t \})$ a.s. as $n\rightarrow \infty$, and by the Dominated Convergence Theorem in the $L^2$ space,  for each $s\in [0,T]$, $\mathbb{E}  \biggl|\varphi^n_\theta( s,x_s | \{ \mu^n_t \})-\varphi^n_\theta(s,x_s | \{ \mu^n_t \}) \biggl|^2 \rightarrow  0.$  Hence,
by  inequality (\ref{ineqcontinuity1}),  $D^2(\tilde{\mu}_t,\tilde{\mu}_t^n) \rightarrow 0$  for any $t \in [0,T]$,  $d_\mathcal{M} \biggl(\{\tilde{\mu}_t\},\{\tilde{\mu}_t^n\} \biggl)  \rightarrow 0 \text{ as } n \rightarrow \infty.$ That is,  $\Gamma $ is continuous. 
    \end{proof}

\begin{proposition}\label{uniq}
Assume \emph{(A1)--(A6)}. Then $\Gamma:\mathcal{M}_{[0,T]}\rightarrow \mathcal{M}_{[0,T]}$ has a fixed point, and \emph{(\ref{MFGbounded1})} has a unique solution.
\end{proposition} 
\begin{proof}
As in the proof in Section 3.2 and the proof of Lemma 5.7 in \cite{Cardaliaguet2013}, the range of the mapping $\Gamma $ is relatively compact, and by Proposition \ref{continuous}, $\Gamma$ is a continuous mapping. Hence, due to the Schauder fixed point theorem~\cite[Theorem 4.1.1]{Smart1980}, $\Gamma$ has a fixed point such that $\Gamma (\{\mu_t\}) = \{\mu_t\} \in \mathcal{M}_{[0,T]}$. By Assumption (A5),  there exists at most one fixed point  \cite{Cardaliaguet2013, LL2007}. Therefore, there exists a unique fixed point solution of flow of probability measures $\{\mu_t^*\}$.   By definition of the solution to a MFG and Proposition \ref{optimization}, the optimal control is also unique.
\end{proof}

\subsection{Proof of main Theorem}

Suppose that $ \biggl((\xi_{\cdot,\theta}^+,\xi_{\cdot,\theta}^- ),\{\mu_{t,\theta} \} \biggl)$ is a solution to (\ref{MFGbounded1}) with a given bound $\theta$, and $x_{t,\theta}$ is the optimally controlled process:
\begin{align*}
dx_{t,\theta} =  \biggl(b(x_{t,\theta}, \mu_{t,\theta}) +\varphi_{1,\theta}(t,x_{t,\theta}|\{\mu_{t,\theta}\}) - \varphi_{2,\theta}(t,x_{t,\theta}|\{\mu_{t,\theta}\})  \biggl)  dt + \sigma dW_t, \quad x_{s,\theta} = x,
\end{align*}
where $\dot{\xi}_{t,\theta}^+-\dot{\xi}_{t,\theta}^- = \varphi_\theta (t,x|\{\mu_{t, \theta}\}) = \varphi_{1,\theta} (t,x |\{\mu_{t,\theta}\}) - \varphi_{2,\theta}(t,x |\{\mu_{t,\theta}\}) $ is the optimal control function.  Note that we explicit write $\mu_{t, \theta}$ here to emphasize the dependence on $\theta$ for the game (MFG-BD).

Given this  $\{\mu_{t,\theta}\}$, let  $v(s,x|\{ \mu_{t,\theta}\}) $ be the value function
of  the stochastic control problem (\ref{Control-FV}), 
and  let $x_{t}$ be the optimal controlled process 
\begin{align*}
dx_{t} =  b(x_{t}, \mu_{t,\theta})dt + \sigma dW_t +d\xi_{t}^+ -d\xi_{t}^- , \quad x_{s-} = x,
\end{align*}
where  the optimal control $\xi_{t}$ is of a feedback form. 
 Hence, denote $$ d \varphi (t,x|\{\mu_{t,\theta}\})=d \varphi_{1} (t,x|\{\mu_{t,\theta}\})-d \varphi_{2} (t,x|\{\mu_{t,\theta}\}) = d\xi_{t}^+ -d\xi_{t}^-$$ as the optimal control function for the stochastic control problem of 
 (\ref{Control-FV}) with the fixed $\{\mu_{t,\theta}\}$. 
Now define
\begin{align*}
dx_{t,\theta}^i &=  \biggl(b(x_{t,\theta}^i, \mu_{t,\theta}) +\varphi_{1,\theta}(t,x_{t,\theta}^i|\{\mu_{t,\theta}\} ) - \varphi_{2,\theta}(t,x_{t,\theta}^i|\{\mu_{t,\theta}\}  )  \biggl)  dt + \sigma dW_{t }^i, \quad x_{s,\theta}^i = x,
\\
dx_{t}^i &=  b(x_{t}^i, \mu_{t,\theta})dt  +d \varphi_{1} (t,x_{t}^i|\{\mu_{t,\theta}\}  )-d \varphi_{2} (t,x_{t}^i|\{\mu_{t,\theta}\} ) + \sigma dW_t^i, \quad x_{s-}^i = x,
\\
dx_{t,\theta}^{i, N}& = \biggl( \frac{1}{N}  \sum_{ j = 1}^N  b_0(x_{t,\theta}^{i, N}, x_{t,\theta}^{j, N} )  +\varphi_{1,\theta}(t,x_{t,\theta}^{i, N}|\{\mu_{t,\theta}\}  ) - \varphi_{2,\theta}(t,x_{t,\theta}^{i, N}|\{\mu_{t,\theta}\}  )  \biggl) dt + \sigma dW_t^i, \quad x_{s,\theta}^{i, N} = x,
\end{align*}

Recall that $(\mu_{t,\theta}, \varphi_\theta)$ is the solution to  (\ref{MFGbounded1}) and $x_{t,\theta}^i$ are i.i.d., and $\mu_{t,\theta}$ is the probability measure of $x_{t,\theta}^i$ for any $i = 1,\ldots, N$. 
We first establish some technical Lemmas. 
\begin{lemma} For any $ 1 \le i \le n$, 
 $ \mathbb{E} \sup_{s\leq t\leq T} \limits    \biggl|x_{t,\theta}^i-x_{t,\theta}^{i,N} \biggl|^2   = O \biggl(\frac{1}{N} \biggl)$.
\label{Nash1}
\end{lemma}
\begin{proof}
 
\begin{align*}
d(x_{t,\theta}^i-x_{t,\theta}^{i,N} ) = \left( \int_\mathbb{R} b_0(x_{t,\theta}^i,y) \mu_{t,\theta} (dy)-\frac{1}{N}\sum_{j=1}^N b_0( x_{t,\theta}^{i,N},x_{t,\theta}^{j,N}) + \varphi_\theta (t, x_{t,\theta}^i|\{\mu_{t,\theta}\} )- \varphi_\theta(t, x_{t,\theta}^{i,N}|\{\mu_{t,\theta}\} )\right)  dt,   
\end{align*}
and
\begin{align*}
d(x_{t,\theta}^i-x_{t,\theta}^{i,N}  )^2  &=  \biggl\lbrace 2 (x_{t,\theta}^i-x_{t,\theta}^{i,N}  )  \biggl(\int_\mathbb{R} b_0(x_{t,\theta}^i,y) \mu_{t,\theta} (dy) \\ &\quad  -\frac{1}{N}\sum_{j=1}^N b_0(x_{t,\theta}^{i,N} ,x_{t,\theta}^{j,N} ) + \varphi_\theta (t, x_{t,\theta}^i|\{\mu_{t,\theta}\} )- \varphi_\theta(t, x_{t,\theta}^{i,N}|\{\mu_{t,\theta}\}  ) \biggl)   \biggr\rbrace dt.  
\end{align*}
By Assumption (A6), $(x_{t,\theta}^i-x_{t,\theta}^{i,N} )  \biggl(\varphi_\theta (t, x_{t,\theta}^i|\{\mu_{t,\theta}\} )- \varphi_\theta(t, x_{t,\theta}^{i,N}|\{\mu_{t,\theta}\}  ) \biggl ) \leq 0$. Consequently, for any $t\in[s,T]$,

\begin{align*}
  |x_{t,\theta}^i-x_{t,\theta}^{i,N}|^2  &    
  \leq  \int_s^t 2  | x_{u,\theta}^i-{x_{u,\theta}^{i,N}}| \biggl| \int_\mathbb{R} b_0(x_{u,\theta}^i,y) \mu_{u,\theta}  (dy)-\frac{1}{N}\sum_{j=1}^N b_0( {x_{u,\theta}^{i,N}},x_{u,\theta}^{j,N}) \biggl|du
  \\ & \leq\int_s^t2|x_{u,\theta}^i-x_{u,\theta}^{i,N}|\biggl| \int_\mathbb{R} b_0(x_{u,\theta}^i,y) \mu_{u,\theta} (dy)-\frac{1}{N}\sum_{j=1}^N b_0( x_{u,\theta}^{i},x_{u,\theta}^j) \biggl|du
  \\& \quad +\int_s^t2|x_{u,\theta}^i-x_{u,\theta}^{i,N}|\biggl| \frac{1}{N}\sum_{j=1}^N b_0( x_{u,\theta}^{i},x_{u,\theta}^j) - \frac{1}{N}\sum_{j=1}^N b_0( x_{u,\theta}^{i},x_{u,\theta}^{j,N}) \biggl|du
  \\& \quad +\int_s^t2|x_{u,\theta}^i-x_{u,\theta}^{i,N}|\biggl| \frac{1}{N}\sum_{j=1}^N b_0( x_{u,\theta}^{i},x_{u,\theta}^{j,N}) - \frac{1}{N}\sum_{j=1}^N b_0( x_{u,\theta}^{i,N},x_{u,\theta}^{j,N}) \biggl|du
  \\& \leq \int_s^t2|x_{u,\theta}^i-x_{u,\theta}^{i,N}|\biggl| \int_\mathbb{R} b_0(x_{u,\theta}^i,y) \mu_{u,\theta} (dy)-\frac{1}{N}\sum_{j=1}^N b_0( x_{u,\theta}^{i},x_{u,\theta}^j) \biggl|du
  \\& \quad +\int_s^t2Lip(b_0)|x_{u,\theta}^i-x_{u,\theta}^{i,N}|^2du+\int_s^t\frac{Lip(b_0)}{N}\sum_{j=1}^N2|x_{u,\theta}^i-x_{u,\theta}^{i,N}||x_{u,\theta}^j-x_{t,\theta}^{j,N}|du
  \\& \leq  \int_s^t\biggl| \int_\mathbb{R} b_0(x_{u,\theta}^i,y) \mu_{u,\theta} (dy)-\frac{1}{N}\sum_{j=1}^N b_0( x_{u,\theta}^{i},x_{u,\theta}^j) \biggl|^2du
  \\& \quad +\int_s^t[1+3Lip(b_0)]|x_{u,\theta}^i-x_{u,\theta}^{i,N}|^2du+\int_s^t\frac{Lip(b_0)}{N}\sum_{j=1}^N|x_{u,\theta}^j-x_{u,\theta}^{j,N}|^2du.
  \end{align*}
By the assumption that the initial distribution among $N$ players is permutation invariant, 
\begin{align*}
  \mathbb{E}  |x_{t,\theta}^{i }-x_{t,\theta}^{i}|^2 
  \leq  & [1+ 4Lip(b_0)] \mathbb{E} \int_s^t |  x_{u,\theta}^{i }-x_{u,\theta}^{i,N}|^2 du\\& + \mathbb{E} \int_s^t   \biggl| \int_\mathbb{R} b_0(x_{u,\theta}^{i },y) \mu_{u,\theta} (dy)-\frac{1}{N}\sum_{j=1}^N b_0( x_{u,\theta}^{i},x_{u,\theta}^{j}) \biggl|^2du, 
\end{align*}
and $x_{\cdot,\theta}^{i}$'s are now i.i.d.. Due to the boundedness of $b_0$, $$\mathbb{E}    \biggl| \int_\mathbb{R} b_0(x_{t,\theta}^{i },y) \mu_{t,\theta} (dy)-\frac{1}{N}\sum_{j=1}^N b_0( x_{t,\theta}^{i},x_{t,\theta}^{j}) \biggl|^2=\epsilon_N^2=O\left(\frac{1}{N}\right).$$
Consequently, 
\begin{align*}
 \mathbb{E}  | x_{t,\theta}^{i }-x_{t,\theta}^{i,N} |^2   
 &\leq   \mathbb{E}  \int_s^t (1+4Lip(b_0)  )    | x_{u,\theta}^{i}-x_{u,\theta}^{i,N} |^2 du+  \epsilon_N^2 du.  
\end{align*}
By Gronwall's inequality,
$$ \mathbb{E}  | x_{t,\theta}^{i }-x_{t,\theta}^{i,N} |^2\leq   \int_s^t \epsilon_N^2 du  \cdot \mathbb{E}  \biggl[\exp(\int_s^t [1+4Lip(b_0) ] du) \biggl]\leq \epsilon_N^2\cdot T\cdot\exp\left\{T[1+4Lip(b_0)]\right\},
$$
and hence,
\begin{equation*}
\mathbb{E} \sup_{s \leq t \leq T}  |x_{t,\theta}^{i}-x_{t,\theta}^{i,N}|^2 \leq \epsilon_N^2\cdot T\cdot\exp\left\{T[1+4Lip(b_0)]\right\} = O \biggl( \frac{1}{N} \biggl).
\end{equation*}
Therefore, $ \mathbb{E} \sup_{s\leq t\leq T} \limits  |x_{t,\theta}^{i}-x_{t,\theta}^{i,N}|^2 = O\left( \frac{1}{N} \right)$.
\end{proof}

Suppose that the first player chooses a different control   $\xi_t' $ which is of a bounded velocity and all other players $i=2,3,\ldots, N$ choose to stay with the optimal control  $\{\xi_{t,\theta}\}$. Denote   $$d\xi_t' = \dot{\xi}_t' dt = \varphi'(t,x) dt, \quad   \text{ and } \quad  d\xi_{t,\theta} = \dot{\xi}_{t,\theta} dt = \varphi_\theta (t,x|\{\mu_{t,\theta}\} ) dt.$$  

Then the corresponding dynamics for the MFG is 
\begin{align*}
d \tilde{x}_{t,\theta}^1 &=     \biggl( b (\tilde{x}_{t,\theta}^1,\mu_{t,\theta}) + \varphi'(t,\tilde{x}_{t,\theta}^1)   \biggl)  dt + \sigma dW_t^1
\end{align*}
 The corresponding dynamics for $N$-player game are 
\begin{align*}
d\tilde{x}_{t,\theta}^{1,N}  &= \left( \frac{1}{N}\sum_{j=1}^N b_0( \tilde{x}_{t,\theta}^{1,N} ,\tilde{x}_{t,\theta}^{j,N}) +  \varphi'(t,\tilde{x}_{t,\theta}^{1,N})\right)  dt + \sigma dW_t^1,
\\d\tilde{x}_{t,\theta}^{i,N} &= \left( \frac{1}{N}\sum_{j=1}^N b (\tilde{x}_{t,\theta}^{i,N},\tilde{x}_{t,\theta}^{j,N}) +  
\varphi_\theta (t,\tilde{x}_{t,\theta}^{i,N}|\{\mu_{t,\theta}\} )\right)  dt + \sigma dW_t^i, \quad \quad \quad 2 \leq i \leq N.
\end{align*} 
We first show 
\begin{lemma}\label{Lemma-Nash}
$  \sup_{2 \leq i \leq N} \limits \mathbb{E} \sup_{0\leq t\leq T} \limits |x_{t,\theta}^{i,N}-\tilde{x}_{t,\theta}^{i,N}| \leq O \biggl(\frac{1}{\sqrt{N}} \biggl) $.  
\end{lemma}
\begin{proof}
For any $2 \leq i \leq N$, 
\begin{align*}
 d(x_{t,\theta}^{i,N}-\tilde{x}_{t,\theta}^{i,N})    = \left[ \frac{1}{N}\sum_{j=1}^N  \left( b_0(x_{t,\theta}^{i,N},x_{t,\theta}^{j,N})-b_0(\tilde{x}_{t,\theta}^{i,N},\tilde{x}_{t,\theta}^{j,N}) \right) +    \varphi_\theta (t,x_{t,\theta}^{i,N}|\{\mu_{t,\theta}\} )  -\varphi_\theta(t,\tilde{x}_{t,\theta}^{i,N}|\{\mu_{t,\theta}\} )\right]  dt.
\end{align*} 
Because $\varphi_\theta (t,x|\{\mu_{t,\theta}\} )$ is nonincreasing in $x$, 
\begin{align*}
 |x_{T,\theta}^{i,N}-\tilde{x}_{T,\theta}^{i,N}|^2 &\leq   \int_s^T   2 (x_{t,\theta}^{i,N}- \tilde{x}_{t,\theta}^{i,N})\left( \frac{1}{N}\sum_{j=1}^N  \left(b_0(x_{t,\theta}^{i,N},x_{t,\theta}^{j,N})-b_0(\tilde{x}_{t,\theta}^{i,N},\tilde{x}_{t,\theta}^{j,N}) \right)   \right)   dt
\\ &\leq   \int_s^T   2 (x_{t,\theta}^{i,N}-\tilde{x}_{t,\theta}^{i,N})  \frac{1}{N}\sum_{j=1}^N  Lip(b_0) \biggl( |x_{t,\theta}^{i,N} - \tilde{x}_{t,\theta}^{i,N}|+ |x_{t,\theta}^{j,N}-\tilde{x}_{t,\theta}^{j,N}|  \biggl)  dt\allowdisplaybreaks
\\ &\leq  2 Lip(b_0)  \int_s^T    |x_{t,\theta}^{i,N}-\tilde{x}_{t,\theta}^{i,N}|^2 + |x_{t,\theta}^{i,N}-\tilde{x}_{t,\theta}^{i,N}| \frac{1}{N}\sum_{j=1}^N    |x_{t,\theta}^{j,N}-\tilde{x}_{t,\theta}^{j,N}|    dt
\\ &\leq  2 Lip(b_0)  \int_s^T    |x_{t,\theta}^{i,N}-\tilde{x}_{t,\theta}^{i,N}|^2 +  \frac{1}{2N}\sum_{j=1}^N  \biggl( |x_{t,\theta}^{i,N}-\tilde{x}_{t,\theta}^{i,N}|^2+  |x_{t,\theta}^{j,N}-\tilde{x}_{t,\theta}^{j,N}|^2  \biggl) dt 
\\ &\leq    Lip(b_0)  \int_s^T  3  |x_{t,\theta}^{i,N}-\tilde{x}_{t,\theta}^{i,N}|^2 +  \frac{1}{ N}\sum_{j=1}^N  |x_{t,\theta}^{j,N}-\tilde{x}_{t,\theta}^{j,N}|^2   dt  ,
\end{align*}
and 
\begin{align*}
 \sup_{2\leq i\leq N}  \limits &  \mathbb{E} \sup_{s\leq t\leq T} \limits |x_{t,\theta}^{i,N}-\tilde{x}_{t,\theta}^{i,N}|^2 
 \\ & \leq Lip(b_0) \int_s^T [  \sup_{2\leq i\leq N} \limits  \mathbb{E} \sup_{s\leq t'\leq t} \limits 3 |x_{t',\theta}^{i,N}-\tilde{x}_{t',\theta}^{i,N}|^2 \\& \quad \quad + \frac{N-1}{N} \sup_{2\leq j\leq N} \limits \mathbb{E} \sup_{s\leq t' \leq t} \limits |x_{t',\theta}^{j,N}-\tilde{x}_{t',\theta}^{j,N}|^2+ \frac{1}{N}\mathbb{E} |x_{t,\theta}^{1,N}-\tilde{x}_{t,\theta}^{1,N}|^2  ]  dt
\\&=  Lip(b_0) \int_s^T  \left[ \frac{4N-1}{N} \sup_{2\leq i\leq N} \limits \mathbb{E} \sup_{s\leq t' \leq t} \limits   |x_{t',\theta}^{i,N}-\tilde{x}_{t',\theta}^{i,N}|^2 +  \frac{1}{N} \mathbb{E}|x_{t,\theta}^{1,N}-\tilde{x}_{t,\theta}^{1,N}|^2 \right]  dt.
\end{align*}
By  Gronwall's inequality, 
\begin{align*}
 \sup_{2\leq i\leq N} \limits \mathbb{E} \sup_{s\leq t\leq T} \limits |x_{t,\theta}^{i,N}-\tilde{x}_{t,\theta}^{i,N}|^2 \leq Lip(b_0) \int_s^T \frac{1}{N} \mathbb{E} |x_{t,\theta}^{1,N}-\tilde{x}_{t,\theta}^{1,N}|^2 dt \cdot e^{\int_0^T Lip(b_0) \frac{4N-1}{N} dt} =O \left( \frac{1}{N} \right).
\end{align*}
So, $\sup_{2\leq i\leq N} \limits  \mathbb{E} \sup_{s\leq t\leq T} \limits |x_{t,\theta}^{i,N}-\tilde{x}_{t,\theta}^{i,N}|=O \biggl(\frac{1}{\sqrt{N}} \biggl). $
\end{proof}

\begin{proof}[Proof of Main Theorem a)]
By Lemma \ref{Nash1}, for any $2 \le i \le N$, 
$  \sup_{s\leq t\leq T} \limits \mathbb{E} |x_{t,\theta}^{i }-x_{t,\theta}^{i,N} | = O \left(\frac{1}{\sqrt{N}} \right)$, and by the triangle inequality, $\sup_{2 \leq i \leq N} \limits \mathbb{E} \sup_{s\leq t\leq T} \limits |x_{t,\theta}^{i}-\tilde{x}_{t,\theta}^{i,N}| = O(\frac{1}{\sqrt{N}})$.
Therefore, 
\begin{equation*}
\sup_{2 \leq i \leq N} \limits \mathbb{E} \sup_{s\leq t\leq T} \limits |x_{t,\theta}^{ i}-\tilde{x}_{t,\theta}^{i, N}|  + \sup_{1 \leq i \leq N} \limits \mathbb{E} \sup_{s\leq t\leq T} \limits |x_{t,\theta}^{ i}-x_{t,\theta}^{ i,N}| = O \left(\frac{1}{\sqrt{N}} \right).
\end{equation*}
 
Finally, define
 \begin{align*}
 d\bar{x}_{t,\theta}^{1,N}  &= \left( \frac{1}{N}\sum_{j=1}^N b_0( \bar{x}_{t,\theta}^{1,N} ,x_{t,\theta}^{j}) +  \varphi'(t,\bar{x}_{t,\theta}^{1,N})\right)  dt + \sigma dW_t^1,
 \end{align*}
 Since $(x-y)(\varphi'(t,x)-\varphi'(t,y)) \leq 0$ by Assumption (A6), then a similar proof as that for Lemma~\ref{Nash1} shows  
 $\mathbb{E} \sup_{0\leq t\leq T} \limits |\tilde{x}_{t,\theta}^{1,N}-\bar{x}_{t,\theta}^{1,N}| = O \left(\frac{1}{\sqrt{N}} \right)$ and $ \mathbb{E} \sup_{0\leq t\leq T} \limits |\bar{x}_{t,\theta}^{1,N}-\tilde{x}_{t,\theta}^{1 }| = O\left( \frac{1}{\sqrt{N}} \right)$.
 Therefore,
\begin{align*}
&E_{x_{s-,\theta}^{N}}\left[J^{1,N}_{\theta}(s,x_{s-,\theta}^{N},\xi_\cdot^{'+},\xi_\cdot^{'-};\xi_{\cdot,\theta}^{-1}|\{\mu_{t,\theta}\} ) \right]
\\&= \mathbb{E} \left[ \int_s^T \frac{1}{N}  \sum_{j=1}^N f_0(\tilde{x}_{t,\theta}^{1,N}, \tilde{x}_{t,\theta}^{j,N})  + \gamma_1 \varphi'_1(t,\tilde{x}_{t,\theta}^{1,N})+\gamma_2 \varphi'_2(t,\tilde{x}_{t,\theta}^{1,N})  dt \right]
\\&\geq \mathbb{E} \left[ \int_s^T \frac{1}{N}  \sum_{j=1}^N f_0(\tilde{x}_{t,\theta}^{1,N},x_{t,\theta}^{j })  +\gamma_1 \varphi'_1(t,\tilde{x}_{t,\theta}^{1,N})+\gamma_2 \varphi'_2(t,\tilde{x}_{t,\theta}^{1,N})  dt \right] -O\left(\frac{1}{\sqrt{N}} \right)
\\&\geq \mathbb{E} \left[ \int_s^T \frac{1}{N}  \sum_{j=1}^N f_0(\bar{x}_{t,\theta}^{1,N}, x_{t,\theta}^j) +\gamma_1 \varphi'_1(t,\bar{x}_{t,\theta}^{1,N})+\gamma_2 \varphi'_2(t,\bar{x}_{t,\theta}^{1,N})  dt \right] -O\left(\frac{1}{\sqrt{N}} \right)
\\&\geq \mathbb{E} \left[ \int_s^T  \int_\mathbb{R}  f_0 (\tilde{x}_{t,\theta}^{1},y) \mu_{t,\theta} (dy) + \gamma_1 \varphi'_1(t,\tilde{x}_{t,\theta}^{1}) + \gamma_2 \varphi'_2(t,\tilde{x}_{t,\theta}^{1})  dt \right] -O\left(\frac{1}{\sqrt{N}} \right)
\\&\geq \mathbb{E} \left[ \int_s^T   \int_\mathbb{R}  f_0 (x_{t,\theta}^{1} ,y) \mu_{t,\theta} (dy) + \gamma_1 \varphi_{1,\theta} (t,{x}_{t,\theta}^{1}|\{\mu_{t,\theta}\} )+\gamma_2 \varphi_{2,\theta} (t,{x}_{t,\theta}^{1}|\{\mu_{t,\theta}\} )  dt \right] -O\left(\frac{1}{\sqrt{N}} \right)
\\& = \mathbb{E} \left[ \int_s^T  \frac{1}{N}  \sum_{j=1}^N f_0(x_{t,\theta}^{1,N} ,x_{t,\theta}^{j,N} )  +\gamma_1 \varphi_{1,\theta} (t,{x}_{t,\theta}^{1,N}|\{\mu_{t,\theta}\} )+\gamma_2 \varphi_{2,\theta} (t,{x}_{t,\theta}^{1,N}|\{\mu_{t,\theta}\} )   dt \right] -O\left(\frac{1}{\sqrt{N}} \right)
\\& = E_{x_{s-,\theta}^{N}}\left[{J^{1,N}_{\theta}}(s,x_{s-,\theta}^N,\xi_{\cdot,\theta}^{ +},\xi_{\cdot,\theta}^{ -};\xi_{\cdot,\theta}^{ -1}|\{\mu_{t,\theta}\} )\right] -O\left(\frac{1}{\sqrt{N}} \right),   
\end{align*}  
where the last inequality is due to the optimality of $\varphi$ for problem (\ref{MFGbounded1}), and the last equality follows a similar proof of Lemma \ref{Nash1}.
\end{proof}
 
\begin{proof}[Proof of Main Theorem b)]
Let all  players except player 1 choose the  optimal controls $(\xi_{\cdot,\theta}^+,\xi_{\cdot,\theta}^-) $, let player one choose any other controls  $(\xi_{\cdot}^{'+},\xi_{\cdot}^{'-}) \in \mathcal{U}$.
Denote $$d \xi_t'= d \varphi ' (t,x )=   d \varphi_1' (t,x )-d \varphi_2' (t,x ), $$
 \begin{align*}
d\tilde{x}_{t}^1 &=   b (\tilde{x}_{t}^1, \mu_{t,\theta} ) dt +d\varphi_1'(t,\tilde{x}_{t}^1)  - d\varphi_2'(t,\tilde{x}_{t}^1 )   + \sigma dW_t^1 \quad \tilde{x}_{s-}^1 = x,\\
d\tilde{x}_{t}^{1,N} &=  \frac{1}{N} \sum_{ j = 1,\ldots, N} b_0(\tilde{x}_{t}^{1,N}, \tilde{x}_{t,}^{j,N} ) dt +d\varphi_1'(t,\tilde{x}_{t}^{1,N} ) -d \varphi_2'(t,\tilde{x}_{t}^{1,N})  + \sigma dW_t^1, \quad \tilde{x}_{s-}^{1,N} = x,  
\\d\tilde{x}_{t}^{i,N} &= \biggl( \frac{1}{N} \sum_{ j = 1,\ldots, N} b_0(\tilde{x}_{t}^{i,N}, \tilde{x}_{t}^{j,N} )  +\varphi_{1,\theta}(t,\tilde{x}_{t}^{i,N}|\{\mu_{t,\theta}\}  ) - \varphi_{2,\theta}(t,\tilde{x}_{t}^{i,N} |\{\mu_{t,\theta}\} )  \biggl) dt + \sigma dW_t^i, \quad x_{s-}^{i,N} = x, 
\\& \text{ for  } i = 2,\ldots,N.
\end{align*}
Then, 
\begin{align*}
 d(x_{t,\theta}^{i,N}-\tilde{x}_{t}^{i,N})    = \left[ \frac{1}{N}\sum_{j=1}^N  \left( b_0(x_{t,\theta}^{i,N},x_{t,\theta}^{j,N})-b_0(\tilde{x}_{t}^{i,N},\tilde{x}_{t}^{j,N}) \right) +    \varphi_\theta(t,x_{t,\theta}^{i,N}|\{\mu_{t,\theta}\} )  -\varphi_\theta (t,\tilde{x}_{t}^{i,N}|\{\mu_{t,\theta}\} )\right]  dt.
\end{align*} 
By definition, $\varphi_\theta (t,x|\{\mu_{t,\theta}\} )$ is nonincreasing in $x$. Hence, a similar  proof to the one for Lemma \ref{Lemma-Nash}  yields

\begin{equation} \label{Lemma-Nash2}
 \sup_{2 \leq i \leq N} \limits \mathbb{E} \sup_{s\leq t\leq T} \limits |x_{t,\theta}^{i,N}-\tilde{x}_{t}^{i,N}| = O \biggl(\frac{1}{\sqrt{N}} \biggl) 
\end{equation} 

From Lemma \ref{Nash1} and  the triangle inequality, $\sup_{2 \leq i \leq N} \limits \mathbb{E} \sup_{s\leq t\leq T} \limits |x_{t,\theta}^{i}-\tilde{x}_{t}^{i,N}| = O \biggl(\frac{1}{\sqrt{N}} \biggl)$.
Therefore, 
\begin{equation*}
\sup_{2 \leq i \leq N} \limits \mathbb{E} \sup_{s\leq t\leq T} \limits |x_{t,\theta}^{ i}-\tilde{x}_{t}^{i, N}|  + \sup_{2 \leq i \leq N} \limits \mathbb{E} \sup_{s\leq t\leq T} \limits |x_{t,\theta}^{ i}-x_{t,\theta}^{ i,N}| = O \left(\frac{1}{\sqrt{N}} \right).
\end{equation*}
Since  $d\varphi'(t,x) $ is also nonincreasing in $x$, then again the same proof as that for Lemma~\ref{Nash1} shows  
 $$\mathbb{E} \sup_{s\leq t\leq T} \limits | \tilde{x}^{ 1 ,N}_{t}-\tilde{x}_{t}^1| = O \left(\frac{1}{\sqrt{N}} \right).$$
By the Lipschitz continuity of $f,f_0$,
\begin{align*}
&E_{x_{s-}^N}\left[J^{1,N} (s,x_{s-}^N ,\xi_\cdot^{'+},\xi_\cdot^{'-};\xi_{\cdot,\theta}^{-1}|\{\mu_{t,\theta}\} ) \right]
\\&= \mathbb{E} \left[ \int_s^T \frac{1}{N}  \sum_{j=1}^N f_0(\tilde{x}_{t}^{1, N}, \tilde{x}_{t}^{j, N})dt +\gamma_1 d\varphi'_1(t,\tilde{x}_{t}^{1, N} )  + \gamma_2 d\varphi'_2(t,\tilde{x}_{t}^{1, N} )  \right]
\\&\geq \mathbb{E} \left[ \int_s^T \frac{1}{N}  \sum_{j=1}^N f_0(\tilde{x}_{t}^{1, N} , x_{t,\theta}^{ j  })dt  +\gamma_1 d\varphi'_1(t,\tilde{x}_{t}^{1, N} )  + \gamma_2 d\varphi'_2(t,\tilde{x}_{t}^{1, N} )   \right] -O\left(\frac{1}{\sqrt{N}} \right)
\\&\geq \mathbb{E} \left[ \int_s^T \int_\mathbb{R} f(\tilde{x}_{t}^{1, N} , y) \mu_{t,\theta}(dy) dt + \gamma_1 d\varphi'_1(t,\tilde{x}_{t}^{1, N})  +\gamma_2 d\varphi'_2(t,\tilde{x}_{t}^{1, N} )   \right] -O\left(\frac{1}{\sqrt{N}} \right)   
\\&\geq \mathbb{E} \left[ \int_s^T \int_\mathbb{R} f(\tilde{x}_{t}^{1 } , y) \mu_{t,\theta} (dy) dt + \gamma_1 d\varphi'_1(t,\tilde{x}_{t}^{1,N} )  +\gamma_2 d\varphi'_2(t,\tilde{x}_{t}^{1,N } )    \right] -O\left(\frac{1}{\sqrt{N}} \right).
\end{align*} 
By definitions of $\tilde{x}_{t }^{1 }$ and $\tilde{x}_{t}^{1,N }$, 
\begin{align}\label{notc}
\begin{split}
& \mathbb{E}  \left| d\varphi'_1(t,\tilde{x}_{t}^{1,N } )    -d\varphi'_1(t,\tilde{x}_{t}^{1 } )-d\varphi'_2(t,\tilde{x}_{t}^{1,N } )    +d\varphi'_2(t,\tilde{x}_{t}^{1 } ) \right|    
\\& \leq \mathbb{E}  d | \tilde{x}_{t}^{1,N } -\tilde{x}_{t}^{1}  | +  \mathbb{E} \left| \frac{1}{N} \sum_{ j = 1,\ldots, N} b_0(\tilde{x}_{t}^{1,N } , \tilde{x}_{t}^{j,N }  )  -  b (\tilde{x}_t^{1 }, \mu_{t,\theta }   ) \right| dt
  =   O\left(\frac{1}{\sqrt{N}} \right),   
  \end{split}
\end{align}
and  by  definition of $\varphi_1',\varphi_2'$, 
\begin{align*}
&\left| \left( d\varphi'_1(t,\tilde{x}_{t}^{1,N }  )    -d\varphi'_1(t,\tilde{x}_{t}^{1 } ) \right)+\left( - d\varphi'_2(t,\tilde{x}_{t}^{1,N }  )    +d\varphi'_2(t,\tilde{x}_{t}^{1} ) \right)\right| 
\\ = &  \left|   d\varphi'_1(t,\tilde{x}_{t}^{1,N }  )    -d\varphi'_1(t,\tilde{x}_{t}^{1 } ) \right|+\left| - d\varphi'_2(t,\tilde{x}_{t}^{1,N }  )    +d\varphi'_2(t,\tilde{x}_{t}^{1 } )\right|. 
\end{align*}
Therefore,
$$  \mathbb{E}  \sup_{s\leq t\leq T} \left|   d\varphi'_1(t,\tilde{x}_{t}^{1,N }  )    -d\varphi'_1(t,\tilde{x}_{t}^{1 } ) \right|  = O\left(\frac{1}{\sqrt{N}} \right),$$    $$ \mathbb{E}  \sup_{s\leq t\leq T}\left| - d\varphi'_2(t,\tilde{x}_{t}^{1,N }  )    +d\varphi'_2(t,\tilde{x}_{t}^{1 } )\right|=  O\left(\frac{1}{\sqrt{N}} \right),$$
and
\begin{align*}
 & \mathbb{E} \left[ \int_s^T \int_\mathbb{R} f(\tilde{x}_{t}^{1 }  , y) \mu_{t,\theta}(dy) dt + \gamma_1 d\varphi'_1(t,\tilde{x}_{t}^{1,N }  )  + \gamma_2 d\varphi'_2(t,\tilde{x}_{t}^{1,N } )   \right] -O\left(\frac{1}{\sqrt{N}} \right)  
   \\   & \geq  \mathbb{E} \left[ \int_s^T \int_\mathbb{R} f(\tilde{x}_{t}^{1  } , y) \mu_{t,\theta} (dy) dt  +\gamma_1 d\varphi'_1(t,\tilde{x}_{t}^{1}  )  + \gamma_2  d\varphi'_2(t,\tilde{x}_{t}^{1 } )   \right] -O\left(\frac{1}{\sqrt{N}} \right)
   \\   & \geq  \mathbb{E} \left[ \int_s^T \int_\mathbb{R} f(x_{t}^{1 } , y) \mu_{t,\theta }  (dy) dt  + \gamma_1 d\varphi_{1} (t,x_{t}^{1 } |\{\mu_{t,\theta}\}  )  + \gamma_2 d\varphi_{2}(t,x_{t}^{1 } |\{\mu_{t,\theta}\}  )   \right] -O\left(\frac{1}{\sqrt{N}} \right)
   \\ & = v (s,x |\{\mu_{t,\theta }  \}) -O\left(\frac{1}{\sqrt{N}} \right).
\end{align*}
The last inequality is due to the optimality of $\varphi$.

Now, by Theorem \ref{thetainfty},  $$ \biggl|v_{\theta} (s,x |\{\mu_{t,\theta}\}) - v (s,x |\{\mu_{t,\theta} \}) \biggl| \le \epsilon_\theta.$$ 
Hence, by $   \mathbb{E} \sup_{s\leq t\leq T} \limits |x_{t,\theta }^{i} -x_{t,\theta}^{i,N } | = \epsilon_N $  and by the analysis  as in the previous steps
\begin{align*}
&E_{x_{s-}^N}\left[J^{1,N} (s,x_{s-}^N ,\xi_\cdot^{'+},\xi_\cdot^{'-};\xi_{\cdot,\theta}^{-1}|\{\mu_{t,\theta}\} ) \right] = E_{x_{s-}^N}[v(s,x_{s-}^N| \{\mu_{t,\theta }    \})] -\epsilon_N  
 \\ & \geq E_{x_{s-}^N}[v_{\theta} ( s,x_{s-}^N|\{\mu_{t,\theta}  \})] - (\epsilon_N+ \epsilon_\theta ) 
\\&= \mathbb{E} \left[ \int_s^T \int_\mathbb{R} f(x_{t,\theta}^{1  }  , y) \mu_{t,\theta}  (dy) dt  + \gamma_1 d\varphi_{1,\theta}(t,x_{t,\theta}^{1 }|\{\mu_{t,\theta}\}  )  + \gamma_2 d\varphi_{2,\theta } (t,x_{t,\theta}^{1 }|\{\mu_{t,\theta}\}  )   \right] -(\epsilon_N+ \epsilon_\theta ) 
\\& \geq  \mathbb{E} \left[ \int_s^T  \frac{1}{N} \sum_{j = 1}^N f_0(x_{t,\theta}^{1,N }  , x_{t,\theta}^{j,N} )  dt  +\gamma_1 d\varphi_{1,\theta} (t,x_{t,\theta}^{1,N } |\{\mu_{t,\theta}\}  )  + \gamma_2 d\varphi_{2,\theta}(t,x_{t,\theta}^{1,N } |\{\mu_{t,\theta}\}  )   \right] -(\epsilon_N+ \epsilon_\theta ) 
\\& =E_{x_{s-}^N}\left[J^{1,N} (s,x_{s-}^N ,\xi_{\cdot,\theta}^{ +},\xi_{\cdot,\theta}^{ -};\xi_{\cdot,\theta}^{ -1}|\{\mu_{t,\theta}\} )\right]-(\epsilon_N+ \epsilon_\theta ) . 
\end{align*}
\end{proof}

\section{Conclusion and discussion}
In this paper, we study the approximation of $N$-player stochastic games with singular controls by a proper model of MFGs with singular control of bounded velocity. In particular, under a set of strategies derived from the MFG solution, the corresponding game value of the $N$-player game with singular controls will deviate from that under NE strategies by at most an error term $\epsilon$; for $N$-player games with singular controls of bounded velocity, this error term $\epsilon = \epsilon_N$ solely depends on the number of players $N$ and $\epsilon_N=O\left(\frac{1}{\sqrt{N}}\right)$; with singular controls of finite variation, this error term $\epsilon$ can be decomposed into $\epsilon=\epsilon_N+\epsilon_\theta$, where $\epsilon_N=O\left(\frac{1}{\sqrt{N}}\right)$ and $\epsilon_\theta$ will vanish when the velocity bound $\theta$ tends to infinity. This finding enriches the literature on the relation between MFGs and $N$-players games in terms of how well MFG models could approximate the corresponding $N$-player games, even when the control processes are not continuous. 

We also notice that there is another direction of approximation one could study: starting from NEs of $N$-player games, whether they will converge to the MFG solutions as $N$ tends to infinity. There have been some works in this direction. For instance, it was shown in \cite{Lacker2016} that the $N$-player open-loop NEs could converge to the mean-field limit in a weak sense of mixed mean-field equilibria; subsequently in \cite{Lacker2020} a closed-case was considered. In \cite{Card2017}, the NE to the $N$-player game was seen as the solution to a system of coupled-HJB equations and its limit as a mean-field system with local coupling was analyzed in terms of propagation of chaos. A special case of time games was studied in \cite{NMT2020} where both the $N$-player game and the mean field game exhibit multiple NEs; it pointed out a transversality condition playing an important role for the mean-field system being the limit of the $N$-player game; concurrently, \cite{CPFP2019} also studied this convergence issue without uniqueness. More recently, \cite{LT2020} studied this convergence for both non-cooperative and cooperative game through propagation of chaos. The majority of the existing works consider the case of continuous controls. It remains to be explored what would happen when non-continuous controls are allowed. 

\end{document}